\documentclass{amsart}
\usepackage{amssymb,color}
\usepackage{tikz}
\theoremstyle{plain}
\newtheorem{theorem}{Theorem}[section]
\newtheorem{cor}{Corollary}[section]
\newtheorem{lem}{Lemma}[section]
\newtheorem{prop}{Proposition}[section]

\theoremstyle{definition}

\newtheorem{exmp}{Example}[section]
\theoremstyle{remark}
\newtheorem{rem}{Remark}[section]

\allowdisplaybreaks

\long\def\symbolfootnote[#1]#2{\begingroup
	\def\thefootnote{\fnsymbol{footnote}}\footnote[#1]{#2}\endgroup}
\begin{document}
	\title[Nonlinear nonhomogeneous Robin problems]
	{Nonlinear nonhomogeneous Robin problems with almost critical and partially concave reaction}
	\author[N.S. Papageorgiou, D.D. Repov\v{s}, and C. Vetro]{N.S. Papageorgiou, D.D. Repov\v{s}, and C. Vetro}
	\address[N.S. Papageorgiou]{National Technical University, Zografou campus, 15780 Athens, Greece
	\& Institute of Mathematics, Physics and Mechanics,
		1000 Ljubljana, Slovenia}
	\email{npapg@math.ntua.gr}
	\address[D.D. Repov\v{s}]{Faculty of Education and Faculty of Mathematics and Physics,
		University of Ljubljana
		\& Institute of Mathematics, Physics and Mechanics,
		1000 Ljubljana, Slovenia}
	\email{dusan.repovs@guest.arnes.si}
	\address[C. Vetro]{Department of Mathematics and Computer Science, University of Palermo, Via Archirafi 34, 90123 Palermo, Italy}
	\email{calogero.vetro@unipa.it}	
	\maketitle	
	\begin{abstract}
		We consider a nonlinear Robin problem driven by a nonhomogeneous differential operator, with reaction which exhibits the competition of two Carath\'eodory terms. One is parametric, $(p-1)$-sublinear with a partially concave nonlinearity near zero. The other is $(p-1)$-superlinear and has almost critical growth. Exploiting the special geometry of the problem, we prove a bifurcation-type result, describing the changes in the set of positive solutions as the parameter $\lambda>0$ varies.
	\end{abstract}
	
	\section{Introduction}
	Let $\Omega \subseteq \mathbb{R}^N$ be a bounded domain with a $C^2$-boundary $\partial \Omega$. In this paper we study the following parametric Robin problem
	\begin{equation}\tag{$P_\lambda$}\label{PL} 
	\begin{cases}
	- \mbox{div}\, a(\nabla u(z)) + \xi(z) |u(z)|^{p-2}u(z)= \lambda f(z,u(z)) + g(z,u(z))  \mbox{ in }\Omega,& \\
	\quad \dfrac{\partial u}{\partial n_a} + \beta (z)|u|^{p-2} u=0  \mbox{ on }\partial \Omega, \, u>0.
	\end{cases}
	\end{equation} 
	In this problem $a: \mathbb{R}^N \to \mathbb{R}^N$ is a continuous, strictly monotone (hence also maximal monotone) map which satisfies certain other regularity and growth conditions listed in hypotheses $H(a)$ below. These conditions are not restrictive and incorporate in our framework many differential operators of interest. We point out that the differential operator $u \to \mbox{div}\, a(\nabla u)$ is not homogeneous and this is a source of difficulties in the analysis of problem \eqref{PL}. The potential function is $\xi \in L^\infty(\Omega)$, $\xi \geq 0$. In the reaction (the right hand side of the equation) $\lambda >0$ is a parameter and $f(z,x)$, $g(z,x)$ are Carath\'{e}odory functions (that is, for all $x \in \mathbb{R}$, $z \to f(z,x), g(z,x)$ are measurable, while for a.a. $z \in \Omega$, $x \to f(z,x), g(z,x)$ are continuous). We assume that $f(z,\cdot)$ is $(p-1)$-superlinear near $0^+$ partially in $z \in \Omega$. So, near zero we have a partially concave nonlinearity and this complicates the geometry of the problem near the origin. Near $+\infty$ for a.a. $z \in \Omega$, $f(z,\cdot)$ is strictly $(p-1)$-sublinear, while for a.a. $z \in \Omega$, $g(z,\cdot)$ exhibits almost critical growth, a fact that further complicates the geometry of the problem, since the embedding of $W^{1,p}(\Omega)$ into $L^{p^\ast}(\Omega)$ is not compact (recall that $p^\ast$ denotes the critical Sobolev exponent corresponding to $1<p<+\infty$, defined by $$p^\ast=\begin{cases}\dfrac{Np}{n-p} & \mbox{if } p<N,\\+\infty & \mbox{if } p \geq N.\end{cases}$$ 
	
	In the boundary condition $\dfrac{\partial u}{\partial n_a}$ denotes the  conormal derivative corresponding to the map $a(\cdot)$ and defined by extension on $W^{1,p}(\Omega)$ of the map $$C^1(\overline{\Omega})\ni u \to (a(\nabla u), n)_{\mathbb{R}^N},$$ with $n(\cdot)$ being the outward unit normal on $\partial \Omega$.  The boundary coefficient is $\beta \in C^{0,\alpha}(\partial \Omega)$ with $\alpha \in (0,1)$ and $\beta \geq 0$. When  $\beta \equiv 0$, we recover the usual Neumann problem. 
	
	We study the nonexistence, existence and multiplicity of positive solutions as the parameter $\lambda>0$ varies. Our main result is a ``bifurcation-type'' theorem, which produces a critical parameter $\lambda^\ast>0$ such that
	\begin{itemize}
		\item[$\bullet$] for all $\lambda \in (0, \lambda^\ast)$ problem \eqref{PL} has at least two positive smooth solutions;
		\item[$\bullet$] for all $\lambda=\lambda^\ast$ problem \eqref{PL} has at least one positive solution;
		\item[$\bullet$] for all $\lambda>\lambda^\ast$ problem \eqref{PL} has no positive solutions.			 
	\end{itemize}
	
	Moreover, we show that we can have positive solutions $u_\lambda \in C^1(\overline{\Omega})$ such that 
	$$\|u_\lambda\|_{C^1(\overline{\Omega})}\to 0 \quad \mbox{as }\lambda \to 0^+.$$
	
	Our approach uses critical point theory combined with suitable truncation and comparison techniques to exploit the particular geometry of the problem. 
	
	The study of problems in which in the reaction we have competition phenomena between nonlinearities of different nature (``concave-convex'' problems), was initiated by the seminal paper of Ambrosetti-Brezis-Cerami \cite{Ref2} for semilinear equations driven by the Dirichlet Laplacian. Their work was extended to equations driven by the Dirichlet $p$-Laplacian
	 by Garc\'{\i}a-Azorero-Peral Alonso-Manfredi \cite{Ref5} and Guo-Zhang \cite{Ref9}.
	 In the aforementioned works, the reaction has the following special form
	$$\lambda x^{q-1} +x^{r-1} \quad \mbox{for all $x \geq0$ with $1<q<p<r<p^\ast$}.$$
	
	More general reactions were assumed by de Figueiredo-Gossez-Ubilla \cite{Ref4}, Gasi\'nski-Papageorgiou \cite{Ref7}, Hu-Papageorgiou \cite{Ref10}, and Papageorgiou-Vetro \cite{Ref22} (Dirichlet problems). For nonlinear Neumann and Robin problems we mention
	 related works of Molica Bisci-R\v{a}dulescu \cite{RefG1,RefG2}, Molica Bisci-Repov\v{s} \cite{RefG3,RefG4}, Papageorgiou-R\v{a}dulescu \cite{Ref18,Ref19}, and Papageorgiou-R\v{a}dulescu-Repov\v{s} \cite{Ref21}.	
	
	\section{Mathematical Background - Hypotheses}\label{S2}
	
	Let $X$ be a Banach space. By $X^\ast$ we denote the topological dual of $X$ and by $\langle\cdot,\cdot \rangle$ we denote the duality brackets for the pair $(X^\ast,X)$. Given $\varphi \in C^1(X, \mathbb{R})$, we say that $\varphi$ satisfies the ``Cerami condition'' (the ``$C$-condition'' for short), if the following property holds:
	
	\smallskip
	
	``Every sequence $\{u_n\}_{n \in \mathbb{N}} \subseteq X$ such that $\{\varphi(u_n)\}_{n \in \mathbb{N}} \subseteq \mathbb{R}$ is bounded and $(1 + \|u_n\|_X) \varphi'(u_n) \to 0 
	$ in $X^\ast$  as $ n \to +\infty,$
	admits a strongly convergent subsequence''. 
	
	\smallskip	
	
	This is a compactness-type condition on the functional $\varphi$ and it leads to a deformation theorem from which one can derive the minimax theory of the critical values of $\varphi$. One of the main results in this theory is the so-called ``Mountain Pass Theorem'' which we recall below.
	\begin{theorem}\label{T1}
		If $X$ is a Banach space, $\varphi \in C^1(X, \mathbb{R})$ satisfies the $C$-condition,  $u_0, u_1 \in X$, $\|u_1-u_0\|_X >\rho$, $\max\{\varphi(u_0),\varphi(u_1)\} < \inf\{\varphi(u): \|u-u_0\|_X =\rho \} =m_\rho ,$ and $c= \inf_{\gamma \in \Gamma} \max_{0\leq t \leq 1} \varphi(\gamma(t))$ with $ \Gamma = \{\gamma \in C([0,1],X): \gamma(0)=u_0, \gamma(1)=u_1\},$
		then $c \geq m_\rho$ and $c$ is a critical value of $\varphi$  $($that is, there exists $\widehat{u} \in X$ such that $\varphi'(\widehat{u})=0$, $\varphi(\widehat{u})=c \geq m_\rho)$.
	\end{theorem}	
	
	Consider a function $\vartheta \in C^1(0, \infty)$, $\vartheta(t)>0$ for all $t>0$, which satisfies
	\begin{equation}
	\label{eq1} 0< \widehat{c}\leq \frac{\vartheta'(t)t}{\vartheta(t)}\leq c_0  \mbox{ and }   c_1 t^{p-1} \leq \vartheta(t) \leq c_2(t^{\tau-1}+t^{p-1}) \quad  \mbox{for all $t>0$},  
	\end{equation}
	with $0<c_1,c_2$ and $1 \leq \tau <p< +\infty$.	
	
	Then the hypotheses on the map $y \to a(y)$ involved in the differential operator of problem \eqref{PL} are the following: 
	\medskip
	
	\noindent $H(a)$: $a(y)=a_0(|y|)y$ for all $y \in \mathbb{R}^N$ with $a_0(t)>0$ for all $t>0$ and
	\begin{itemize}
		\item[$(i)$] $a_0 \in C^1(0,\infty)$, $t \to a_0(t)t$ is strictly increasing on $(0,+\infty)$, $a_0(t)t \to 0^+$ as $t \to 0^+$ and $\lim\limits_{t \to 0^+}\dfrac{a_0'(t)t}{a_0(t)}>-1;$
		\item[$(ii)$] there exists $c_3>0$ such that $|\nabla a(y)| \leq c_3 \dfrac{\vartheta(|y|)}{|y|}$ for all  $y \in \mathbb{R}^N\setminus \{0\}$;
		\item[$(iii)$] $(\nabla a (y) \xi ,\xi)_{\mathbb{R}^N} \geq \dfrac{\vartheta(|y|)}{|y|}|\xi |^2$ for all $y \in \mathbb{R}^N\setminus \{0\}$,  $\xi \in \mathbb{R}^N$;
		\item[$(iv)$] If $\widehat{G}_0(t)= \int_0^t a_0(s)s ds$ for all $t>0$, then there exists $q \in (1,p)$ such that \begin{align*}& \limsup_{t \to 0^+} \dfrac{\widehat{G}_0(t)}{t^q} \leq c_\ast \mbox{ with $c_\ast >0$,}\\ & p\,\widehat{G}_0(t)-a_0(t)t^2 \geq0 \mbox{ for all } t \geq 0.\end{align*}
	\end{itemize}	
	
	\begin{rem} Conditions $H(a) \, (i), (ii),(iii)$ are dictated by the nonlinear regularity theory of Lieberman \cite{Ref11} (p. 320) and the nonlinear maximum principle of Pucci-Serrin \cite{Ref23} (pp. 111, 120). These conditions were first used by Papageorgiou-R\v{a}dulescu \cite{Ref16,Ref17}. Condition $H(a) \, (iv)$ serves the needs of our problem, but it is mild and it is satisfied in all cases of interest (see the examples below).
	\end{rem}
	
	These conditions imply that $t \to \widehat{G}_0(t)= \int_0^t a_0(s)s ds$ is strictly convex and strictly increasing. We set $\widehat{G}(y)=\widehat{G}_0(|y|)$ for all $y \in \mathbb{R}^N$. We have that $G(\cdot)$ is convex, $\widehat{G}(0)=0$, and
	$$\nabla \widehat{G}(0)=0, \quad \nabla \widehat{G}(y)=\widehat{G}_0'(|y|) \frac{y}{|y|}=a_0(|y|)y=a(y) \quad \mbox{for all } y \in \mathbb{R}^N\setminus \{0\}.$$
	So, $\widehat{G}(\cdot)$ is the primitive of the map $a(\cdot)$ and on account of the convexity of $\widehat{G}(\cdot)$ and since $\widehat{G}(0)=0$, we have
	\begin{equation}
	\label{eq2}\widehat{G}(y) \leq (a(y),y)_{\mathbb{R}^N} \quad \mbox{for all } y \in \mathbb{R}^N.\end{equation}
	The next lemma summarizes the main properties of the map $a(\cdot)$ and is a straightforward consequence of \eqref{eq1} and hypotheses $H(a) \, (i),(ii),(iii)$.
	
	\begin{lem}\label{Lem2}
		If hypotheses $H(a) \, (i),(ii),(iii)$ hold, then
		\begin{itemize}
			\item[(a)] $y \to a(y)$ is strictly monotone and continuous (thus also maximal monotone);
			\item[(b)] $|a(y)| \leq c_4 (|y|^{\tau-1}+|y|^{p-1})$ for all $y \in \mathbb{R}^N$ and
			 some $c_4>0$;
			\item[(c)] $(a(y),y)_{\mathbb{R}^N} \geq \dfrac{c_1}{p-1}|y|^p$ for all $y \in \mathbb{R}^N$. 
		\end{itemize}
	\end{lem}
	This lemma and \eqref{eq2} lead to the following growth estimates for the primitive $\widehat{G}(\cdot)$.
	
	\begin{cor}\label{cor3}
		If hypotheses $H(a) \, (i),(ii),(iii)$ hold, then  
		$\dfrac{c_1}{p(p-1)}|y|^p \leq \widehat{G}(y) \leq c_5 (|y|^\tau+ |y|^p)$ for all $y \in \mathbb{R}^N$
		and   some $c_5>0.$
	\end{cor}
	
	\begin{exmp}
		The following maps satisfy hypotheses $H(a)$ (for details see Papageorgiou-R\v{a}dulescu \cite{Ref17}):
		\begin{itemize}
			\item[(a)] $a(y)=|y|^{p-2}y$, $1<p<+\infty$. This map corresponds to the $p$-Laplace differential operator defined by
			$\Delta_p u =\mbox{div}\, (|\nabla u |^{p-2} \nabla u) $ for all $u \in W^{1,p}(\Omega).$
			\item[(b)] $a(y)=|y|^{p-2}y+\mu|y|^{q-2}y$, $1<q<p<+\infty$, $\mu \geq 0$. This map corresponds to the $(p,q)$-Laplacian defined by
			$\Delta_p u +\Delta_q u $ for all $u \in W^{1,p}(\Omega).$
			Such operators arise in problems of mathematical physics (see Cherfils-Il$^\prime$yasov \cite{Ref3}). 
			\item[(c)] $a(y)=(1+|y|^2)^{\frac{p-2}{2}}y$, $1<p<+\infty$. This map corresponds to the generalized $p$-mean curvature differential operator defined by
			$\mbox{div}\, ((1+|\nabla u |^{2})^{\frac{p-2}{2}} \nabla u)$ for all $ u \in W^{1,p}(\Omega).$
			\item[(d)] $a(y)=|y|^{p-2}y\left[1+\dfrac{1}{1+|y|^p}\right]$, $1<p<+\infty$. This map corresponds to the following pertubation of the $p$-Laplacian
			$\Delta_p u + \mbox{div}\, \left(\dfrac{|\nabla u |^{p-2}\nabla u}{1+|\nabla u |^{p}}\right) $ for all $ u \in W^{1,p}(\Omega).$ 
		\end{itemize}
	\end{exmp}
	
	Let $A: W^{1,p}(\Omega) \to W^{1,p}(\Omega)^\ast$ be defined by
	\begin{equation}\label{eq3}\langle A(u),h \rangle = \int_\Omega (a(\nabla u), \nabla h)_{\mathbb{R}^N}dz \quad \mbox{for all } u,h \in W^{1,p}(\Omega).\end{equation}
	
	Using Lemma \ref{Lem2}, we obtain the following result concerning the map $A(\cdot)$ (see Gasi\'nski-Papageorgiou \cite{Ref8},  Problem 2.192, p. 279).
	
	\begin{prop}\label{prop4}
		If hypotheses $H(a) \, (i),(ii),(iii)$ hold, then the map \\ $A : W^{1,p}(\Omega) \to W^{1,p}(\Omega)^\ast$ defined by \eqref{eq3} is bounded $($that is, it maps bounded sets to bounded sets$)$, continuous, monotone $($hence also maximal monotone$)$ and of type $(S)_+$ $($that is, if $u_n  \xrightarrow{w} u $ in $W^{1,p}(\Omega)$ and $\limsup_{n \to +\infty} \langle A(u_n),u_n-u \rangle \leq 0$, then $u_n \to u$ in $W^{1,p}(\Omega))$. 
	\end{prop}
	
	The following spaces will play a central role in the study of problem \eqref{PL}: the Sobolev space $W^{1,p}(\Omega)$, the Banach space $C^1(\overline{\Omega})$ and the ``boundary'' Lebesgue space $L^p(\partial \Omega)$. By $\| \cdot \|$ we denote the norm of the Sobolev space $W^{1,p}(\Omega)$ defined by $$\|u \| = \left[ \|u\|^p_p + \|\nabla u \|_p^p \right] ^{1/p} \mbox{ for all $u \in W^{1,p}(\Omega).$}$$
	The Banach space $C^1(\overline{\Omega})$ is ordered with order (positive) cone 
	$C_+=\{u \in C^1(\overline{\Omega}) \, : \, u(z) \geq 0 \mbox{ for all } z \in \overline{\Omega}\}.$
	This cone has a nonempty interior given by
	$$D_+=\left\{u \in C_+ \, : \, u(z) > 0 \mbox{ for all } z \in \overline{\Omega} \right\}.$$

	Also, we will consider another open cone in $C^1(\overline{\Omega})$, namely the cone
	$${\rm int \ }C_+=\left\{u \in C_+ \, : \, u(z) > 0 \mbox{ for all } z \in  \Omega,\, \dfrac{\partial u}{\partial n}\Big|_{\partial \Omega \cap u^{-1}(0)}<0\right\}.$$

On $\partial \Omega$ we consider the $(N-1)$-dimensional surface (Hausdorff) measure $\sigma(\cdot)$. Using this measure, we can define in the usual way the boundary Lebesgue spaces $L^q(\partial \Omega)$, $1 \leq q \leq  +\infty$. There exists a unique continuous linear map $\gamma_0: W^{1,p}(\Omega) \to L^p(\partial \Omega)$, known as the ``trace map'', such that
	$\gamma_0(u)=u \big|_{\partial \Omega} $ for all $u \in W^{1,p}(\Omega) \cap C(\overline{\Omega}).$
	So, the trace map extends the notion of boundary values to all Sobolev functions.  The trace map   $\gamma_0(\cdot)$ is compact into $L^q(\partial \Omega)$ for all $q \in \left[1, \dfrac{(N-1)p}{N-p}\right)$ if $ p< N$ and into $L^q(\partial \Omega)$ for all $1 \leq q < +\infty$ if $N \leq p$. Also, we have $$\mbox{im }\gamma_0=W^{\frac{1}{p'},p}(\partial \Omega) \, \left(\frac{1}{p}+\frac{1}{p'}=1\right), \quad  \mbox{ker }\gamma_0=W^{1,p}_0(\Omega).$$
	In the sequel, for notational economy, we drop the use of the map $\gamma_0(\cdot)$. All restrictions of   Sobolev functions on $\partial \Omega$ are understood in the sense of traces.	
	
	We introduce the following hypotheses on the potential  $\xi(\cdot)$ and the boundary coefficient $\beta(\cdot)$:
	\begin{itemize}
		\item[$H(\xi)$:] $\xi \in L^\infty(\Omega)$, $\xi(z) \geq 0$ for a.a. $z \in \Omega$.
		\item[$H(\beta)$:] $\beta \in C^{0,\eta}(\partial \Omega )$ for some $\eta \in (0,1)$, $\beta(z) \geq0$ for all $z \in \partial \Omega$.
		\item[$H_0$:] $\xi \not \equiv 0$ or $\beta \not \equiv 0$.
	\end{itemize}
	
	\begin{rem}
		If $\beta \equiv 0$, then we have the usual Neumann problem.
	\end{rem}
	
	The next two lemmata can be found in Papageorgiou-R\v{a}dulescu-Repov\v{s} \cite{Ref20}.
	
	\begin{lem}
		\label{Lem5}If $\widehat{\xi} \in L^\infty(\Omega)$, $\widehat{\xi}(z) \geq 0$ for a.a. $z \in \Omega$, $\widehat{\xi} \not \equiv 0$, then there exists $c_6>0$ such that $\|\nabla u\|_p^p + \int_\Omega \widehat{\xi}(z)|u|^p dz \geq c_6 \|u\|^p$ for all $u \in W^{1,p}(\Omega)$.
	\end{lem}
	
	\begin{lem}
		\label{Lem6}If $\widehat{\beta} \in L^\infty(\partial \Omega)$, $\widehat{\beta}(z) \geq 0$ for  a.a. $z \in \partial \Omega$, $\widehat{\beta} \not \equiv 0$, then there exists $c_7>0$ such that $\|\nabla u\|_p^p + \int_{\partial \Omega} \widehat{\beta}(z)|u|^p d\sigma \geq c_7 \|u\|^p$ for all $u \in W^{1,p}(\Omega)$.
	\end{lem}	
	
	Now consider a Carath\'{e}odory function $f_0: \Omega \times \mathbb{R} \to \mathbb{R}$ which satisfies
	$$|f_0(z,x)| \leq a_0(z)(1+|x|^{r-1}) \quad \mbox{for a.a. } z \in \Omega \mbox{ and all } x \in \mathbb{R},$$
	with $a_0 \in L^\infty(\Omega)$, $1<r \leq p^*$. We set $F_0(z,x)=\int_0^xf_0(z,s)ds$ and consider the $C^1$-functional $\varphi_0 :W^{1,p}(\Omega) \to \mathbb{R}$ defined by 
	$$\varphi_0(u)=\int_\Omega \widehat{G}(\nabla u)dz + \frac{1}{p}\int_{\partial \Omega} \beta(z) |u|^p d\sigma -\int_\Omega F_0(z,u) dz \quad \mbox{for all } u \in W^{1,p}(\Omega).$$
	
	The next result is an outgrowth of the nonlinear regularity theory and can be found in Papageorgiou-R\v{a}dulescu \cite{Ref16}.
	\begin{prop}\label{P7}
		If hypotheses $H(a)$, $H(\beta)$ hold and $u_0 \in W^{1,p}(\Omega)$ is a local $C^1(\overline{\Omega})$-minimizer of $\varphi_0(\cdot)$, that is, there exists $\rho_1>0$ such that $\varphi_0(u_0) \leq \varphi_0(u_0+h)$ for all $h \in C^1(\overline{\Omega})$, $\|h\|_{C^1(\overline{\Omega})}\leq \rho_1,$
		then $u_0 \in C^{1,\alpha}(\overline{\Omega})$ for some $\alpha \in (0,1)$ and it is also a local $W^{1,p}(\Omega)$-minimizer of $\varphi_0(\cdot)$, that is, there exists $\rho_2 >0$ such that
		$\varphi_0(u_0) \leq \varphi_0(u_0+h)$ for all $h \in W^{1,p}(\Omega)$, $\|h\|\leq \rho_2.$
	\end{prop}
	
	This result is a powerful tool in the study of elliptic problems, when it is combined with the following strong comparison principle due to Papageorgiou-R\v{a}dulescu-Repov\v{s} \cite{Ref20}.
	
	\begin{prop}
		\label{P8} If hypotheses $H(a)$ hold, $\widehat{\xi} \in L^\infty(\Omega)$, $\widehat{\xi}(z) \geq 0$ for a.a. $z \in \Omega$, $h_1,h_2 \in L^\infty(\Omega)$ such that $0< c_8 \leq h_2(z)-h_1(z)$ for a.a. $z \in \Omega$, $u,v \in C^1(\overline{\Omega}) \setminus \{0\}$ satisfy $u \leq v$ and 
		\begin{align*}& - {\rm div \ }a(\nabla u(z))+\widehat{\xi}(z) |u(z)|^{p-2}u(z)=h_1(z) \quad \mbox{for a.a. } z \in \Omega,\\ & - {\rm div \ }a(\nabla v(z))+\widehat{\xi}(z) |v(z)|^{p-2}v(z)=h_2(z) \quad \mbox{for a.a. } z \in \Omega,\end{align*} then $v-u \in {\rm int \ } \widehat{C}_+$.
	\end{prop}
	
	Next, let us fix some basic notation which we will use in the sequel. So, for $x \in \mathbb{R}$, we set $x^\pm = \max \{\pm x, 0\}$. Then for $u \in W^{1,p}(\Omega)$, we define $u^\pm(\cdot)=u(\cdot)^\pm$ and we know that $$u^\pm \in W^{1,p}(\Omega),  \quad u=u^+-u^-, \quad |u| =u^++u^-.$$
	
	If $k: \Omega \times \mathbb{R} \to \mathbb{R}$ is a measurable function (for example, a Carath\'eodory function), then we set $N_k(u)(\cdot)=k(\cdot,u(\cdot))$ for all $u \in W^{1,p}(\Omega)$ (the Nemytskii operator corresponding to $k(\cdot,\cdot)$). Also, by $|\cdot |_N$ we denote the Lebesgue measure on $\mathbb{R}^N$. Given $u,v \in W^{1,p}(\Omega)$ with $u \leq v$, we can define the order interval $[u,v]$ by setting 
	$$[u,v]=\{y \in W^{1,p}(\Omega): u(z) \leq y(z) \leq v(z) \mbox{ for a.a. } z \in \Omega \}.$$
	
	By ${\rm int}_{C^1(\overline{\Omega})}[u,v]$, we denote the interior in $C^1(\overline{\Omega})$ of $[u,v] \cap C^1(\overline{\Omega})$. Also, if $u \in W^{1,p}(\Omega)$, then
	$$[u)=\left\{y \in W^{1,p}(\Omega): u(z) \leq y(z) \mbox{ for a.a. } z \in \Omega\right\}.$$

	If $X$ is a Banach space and $\varphi \in C^1(X, \mathbb{R})$, then by $K_\varphi$ we denote the critical set of $\varphi$, that is, $K_\varphi  = \{u \in X : \varphi'(u) =0 \}$. 
	
	Finally, we introduce the hypotheses on the two competing functions in the reaction of problem \eqref{PL}.
	
	\smallskip
	\noindent $H(f)$: $f: \Omega \times \mathbb{R} \rightarrow \mathbb{R}$ is a Carath\'{e}odory function such that $f(z,0) =0$ for a.a. $z \in \Omega$ and
	\begin{itemize}
		\item[(i)]  for every $\rho >0$, there exists $a_\rho \in L^\infty(\Omega)$ such that $f(z, x) \leq a_\rho (z)$ for a.a. $z \in \Omega$
		and all $0 \leq x \leq \rho$;\\
		\item[(ii)] $f(z, x) \geq \eta_s>0$ for a.a. $z \in \Omega$, all $x \geq s >0$, and   $\lim\limits_{x \to +\infty} \dfrac{f(z,x)}{x^{p-1}}=0$ uniformly  for a.a. $z \in \Omega$;\\
		\item[(iii)] there exist $\mathcal{U} \subseteq \Omega$ open and $\delta_0 \in (0,1]$ such that $\overline{\mathcal{U}} \subseteq \Omega$ and $c_9 x^{q-1} \leq f(z,x)$ for a.a. $z \in \overline{\mathcal{U}}$, all $0 \leq x \leq \delta_0$ with $c_9>0$, and $q \in (1,p)$ as in $H(a)\, (iv)$.
	\end{itemize}
	
	\begin{rem}
		Since we are looking for positive solutions and all the above hypotheses concern the positive semiaxis, we may assume without any loss of generality, that $f(z,x)=0$ for a.a. $z \in \Omega$
		and all $x \leq 0$. Hypothesis $H(f)\, (ii)$ implies that for a.a. $z \in \Omega$, $f(z,\cdot)$ is strictly $(p-1)$-sublinear near $+\infty$. Hypothesis $H(f)\,(iii)$ implies that there is a partially concave nonlinearity near zero.
	\end{rem}
	
	\smallskip
	\noindent $H(g)$: $g: \Omega \times \mathbb{R} \rightarrow \mathbb{R}_+=[0,+\infty)$ is a Carath\'{e}odory function such that $g(z,0) =0$ for a.a. $z \in \Omega$ and
	\begin{itemize}
		\item[(i)]  for every $\rho >0$, there exists $a_\rho \in L^\infty(\Omega)$ such that $g(z, x) \leq a_\rho (z)$ for a.a. $z \in \Omega$
		and  all $0 \leq x \leq \rho$;\\
		\item[(ii)]    $\lim\limits_{x \to +\infty} \dfrac{g(z,x)}{x^{p^\ast-1}}=0$ and $\lim\limits_{x \to +\infty} \dfrac{g(z,x)}{x^{p-1}}=+\infty$ uniformly  for a.a. $z \in \Omega$;\\
		\item[(iii)]  $\lim\limits_{x \to 0^+} \dfrac{g(z,x)}{x^{p-1}}=0$ uniformly  for a.a. $z \in \Omega$.
	\end{itemize}
	
	\begin{rem}
		Again we may assume that $g(z,x)=0$ for a.a. $z \in \Omega$
		and
		 all $x \leq 0$.
		Hypothesis $H(g)\, (ii)$ implies that for a.a. $z \in \Omega$, $g(z,\cdot)$ is $(p-1)$-superlinear and has almost critical growth. Hypothesis $H(g)\, (iii)$ says that for a.a. $z \in \Omega$, $g(z,\cdot)$ is  $(p-1)$-sublinear near zero, in contrast to $f(z, \cdot)$ which exhibits a partially concave nonlinearity.
	\end{rem}	
	
	Usually superlinear problems are treated using the so-called Ambrosetti-Rabinowitz condition (see, for example, Motreanu-Motreanu-Papageorgiou \cite{Ref13}, p. 341). This condition, although useful in checking the compactness condition for the energy (Euler) functional of the problem, it is rather restrictive. For this reason we employ a weaker condition (see hypothesis $\widehat{H}_0\, (i)$ below), which incorporates in our framework also superlinear terms which have ``slower'' growth near $+\infty$ and fail to satisfy the Ambrosetti-Rabinowitz condition.
	
	We introduce $F(z,x)=\int_0^x f(z,s)ds$ and $G(z,x)=\int_0^x g(z,s)ds$. For every $\lambda >0$ we define
	$$e_\lambda(z,x)=[\lambda f(z,x)+g(z,x)]x-p[\lambda F(z,x)+G(z,x)].$$
	
	\medskip
	\noindent $\widehat{H}_0$: for all $\lambda$ in a bounded set $B \subseteq (0,+\infty)$, we have:
	\begin{itemize}
		\item[(i)]  there exists $\eta_B \in L^1(\Omega)$ such that $e_\lambda(z, x) \leq e_\lambda(z,v)+\eta_B(z)$ for a.a. $z \in \Omega$
		and  all $0 \leq x \leq v$, $\lambda \in B$;\\
		\item[(ii)]  for every $\rho >0$, we can find  $\widehat{\xi}^B_\rho >0$ such that for a.a. $z \in \Omega$
		and  all $\lambda \in B$,
		$$x \to \lambda f(z,x)+g(z,x)+ \widehat{\xi}^B_\rho x^{p-1}$$ is nondecreasing on $[0,\rho]$.
	\end{itemize}
	
	\begin{rem}
		Hypothesis $\widehat{H}_0\, (i)$ replaces the Ambrosetti-Rabinowitz condition. It is a slight generalization of a condition used by Li-Yang \cite{Ref12} (see also Mugnai-Papageorgiou \cite{Ref14}). Hypothesis $\widehat{H}_0 \ (ii)$ is satisfied, if, for example, for a.a. $z \in \Omega$, the functions $f(z,\cdot)$, $g(z,\cdot)$ are differentiable and for every $\rho >0$, there exists $\widehat{\xi}^B_\rho >0$ such that $$[ \lambda f_x^\prime(z,x)+g^\prime_x(z,x)]x^2 \geq - \widehat{\xi}^B_\rho |x|^{p} \quad \mbox{for a.a. $z \in \Omega$
		and
		 all $0 \leq x \leq \rho$, 
		  $\lambda \in B$}.$$
	\end{rem}
	
	\begin{exmp}
		The following pair of functions $f(z,x)$, $g(z,x)$ satisfies hypotheses $H(f)$, $H(g)$, $\widehat{H}_0$ above:
		$$f(z,x)=\widehat{a}(z)x^{q-1}+c_{10}x^{\tau-1}$$ with $\widehat{a} \in L^\infty(\Omega)_+ \cap {\rm int \, }L^\infty(\mathcal{U})_+$ with $\mathcal{U} \subseteq \Omega$ open, $\overline{\mathcal{U}} \subseteq \Omega$, $c_{10}>0$, $\tau <p$, and $$g(z,x)=\mu(z)x^{p-1} \ln(1+x)$$ with $\mu \in L^\infty(\Omega)$, $\mu(z) \geq \gamma >0$ for a.a. $z \in \Omega$. The function $g(z,\cdot)$ does not satisfy the Ambrosetti-Rabinowitz condition.
	\end{exmp}
	
	In what follows for the sake of simplicity, the collection of all the hypotheses on the data of \eqref{PL}, namely the hypotheses $H(a)$, $H(\xi)$, $H(\beta)$, $H_0$, $H(f)$, $H(g)$, $\widehat{H}_0$ will be denoted by $\widetilde{H}$.
	
	\section{A Bifurcation-Type Theorem}\label{S3}
	
	We introduce the following two sets:
	\begin{align*}
		\mathcal{L} & =\{\lambda >0 : \mbox{ problem \eqref{PL} admits a positive solution}\},\\ S(\lambda) & = \mbox{ set of positive solutions of  \eqref{PL} ($\lambda >0$)}.
	\end{align*}
	
	\begin{prop}
		\label{P9} If hypotheses $\widetilde{H}$ hold, then $S(\lambda) \subseteq D_+$ for all $\lambda >0$.
	\end{prop}
	
	\begin{proof}
		Of course the result is trivially true if $S(\lambda)= \emptyset$. 
		
		So, suppose that $S(\lambda) \neq \emptyset$ and let $u \in S(\lambda)$. Then
		\begin{equation}\label{eq4} 
		\begin{cases}
		- \mbox{div}\, a(\nabla u(z)) + \xi(z) u(z)^{p-1}= \lambda f(z,u(z)) + g(z,u(z))  \mbox{ for a.a. }z \in \Omega,& \\
		\quad \dfrac{\partial u}{\partial n_a} + \beta (z)u^{p-1}=0  \mbox{ on }\partial \Omega,
		\end{cases}
		\end{equation} 
		(see Papageorgiou-R\v{a}dulescu \cite{Ref15}).
		
		From \eqref{eq4} and Proposition 7 of Papageorgiou-R\v{a}dulescu \cite{Ref16}, we have $u \in L^\infty(\Omega).$
		
		Apply the regularity theory of Lieberman \cite{Ref11} (p. 320), to obtain that 
		$$u \in C^{1,\gamma}(\overline{\Omega}) \quad \mbox{for some } \gamma \in (0,1).$$
		
		Let $\rho=\|u\|_{C^1(\overline{\Omega})}$, $B=\{\lambda\}$ and let $\widehat{\xi}^B_\rho>0$ be as postulated by hypothesis $\widehat{H}_0(ii)$. We have
		\begin{align*}
			& - {\rm div \ } a(\nabla u(z)) +[\xi(z)+\widehat{\xi}^B_\rho]u(z)^{p-1} \geq 0 \quad \mbox{for a.a. } z \in \Omega,\\ \Rightarrow \quad & {\rm div \ } a(\nabla u(z)) \leq [\|\xi\|_\infty+\widehat{\xi}^B_\rho]u(z)^{p-1} \quad \mbox{for a.a. } z \in \Omega.
		\end{align*}
		
		Using the nonlinear maximum principle of Pucci-Serrin \cite{Ref23} (Theorem 5.4.1, p. 111), we have $$u(z) >0 \quad \mbox{for all } z \in \Omega.$$
		
		Finally, invoking the Boundary Point Lemma of Pucci-Serrin \cite{Ref23} (Theorem 5.5.1, p. 120), we conclude that $u \in D_+$.
		
		Therefore for every $\lambda>0$, $S(\lambda) \subseteq D_+$.
	\end{proof}
	
	Next, we show the nonemptiness of $\mathcal{L}$.
	
	\begin{prop}
		\label{P10} If hypotheses $\widetilde{H}$ hold, then $\mathcal{L} \neq \emptyset$.
	\end{prop}
	
	\begin{proof}
		Let $\eta>0$ and consider the following auxiliary Robin problem 	
		\begin{equation} \label{eq5} 
		\begin{cases}
		- \mbox{div}\, a(\nabla u(z)) + \xi(z) |u(z)|^{p-2}u(z)= \eta  \mbox{ in }\Omega,& \\
		\quad \dfrac{\partial u}{\partial n_a} + \beta (z)|u|^{p-2} u=0  \mbox{ on }\partial \Omega.
		\end{cases}
		\end{equation}
		We introduce the operator $V: W^{1,p}(\Omega) \to W^{1,p}(\Omega)^\ast$ defined by
		$$\langle V(u),h \rangle = \langle A(u),h \rangle + \int_\Omega \xi(z)|u|^{p-2}uhdz +$$ $$  \int_{\partial \Omega} \beta(z)|u|^{p-2}uhd\sigma \quad \mbox{for all } u,h \in W^{1,p}(\Omega),$$
		which is continuous, monotone (see Proposition \ref{prop4}), hence it is also maximal monotone. Also, we have 
		\begin{align*} \langle V(u),h \rangle & \geq \frac{c_1}{p-1}\|\nabla u \|^p_p + \int_\Omega \xi(z)|u|^{p}dz +  \int_{\partial \Omega} \beta(z)|u|^{p}d\sigma \quad \mbox{(see Lemma \ref{Lem2})}\\ & \geq c_{11}\|u\|^p \quad \mbox{for some $c_{11}>0$ (see Lemmata \ref{Lem5} and \ref{Lem6})},\\ \Rightarrow \quad & V(\cdot) \mbox{ is coercive}. 
		\end{align*}
		
		A maximal monotone coercive operator is surjective (see Gasi\'nski-Papa\-geor\-giou \cite{Ref6}, Corollary 3.2.31, p. 319). So, we can find $\overline{u} \in W^{1,p}(\Omega)$, $\overline{u}\neq 0$ such that
		\begin{align} \nonumber & V(\overline{u})=\eta, \\ \label{eq6} \Rightarrow \quad &  \langle A(\overline{u}),h \rangle + \int_\Omega \xi(z)|\overline{u}|^{p-2}\overline{u}hdz +  \int_{\partial \Omega} \beta(z)|\overline{u}|^{p-2}\overline{u}hd\sigma = \eta  \int_\Omega h dz \\ & \nonumber \hskip 8cm\mbox{for all } h \in W^{1,p}(\Omega).
		\end{align}
		In \eqref{eq6} we choose $h=-u^- \in W^{1,p}(\Omega)$. Then using Lemma \ref{Lem2}, we obtain 
		\begin{align*} &
			\frac{c_1}{p-1}\|\nabla \overline{u}^- \|^p_p + \int_\Omega \xi(z)(\overline{u}^-)^{p}dz +  \int_{\partial \Omega} \beta(z)(\overline{u}^-)^{p}d\sigma \leq 0,\\ \Rightarrow \quad & c_{12} \|\overline{u}\|^p \leq 0 \quad \mbox{for some $c_{12} >0$ (see Lemmata \ref{Lem5} and \ref{Lem6}),}\\ \Rightarrow \quad & \overline{u} \geq 0, \, \overline{u} \neq 0.
		\end{align*}	
		
		From \eqref{eq6} we obtain
		\begin{equation*}  
			\begin{cases}
				- \mbox{div}\, a(\nabla  \overline{u}(z)) + \xi(z)  \overline{u}(z)^{p-1}= \eta  \mbox{ for a.a.  }z \in \Omega,& \\
				\quad \dfrac{\partial  \overline{u}}{\partial n_a} + \beta (z) \overline{u}^{p-1} =0  \mbox{ on }\partial \Omega.
			\end{cases}
		\end{equation*}
		(see Papageorgiou-R\v{a}dulescu \cite{Ref15}). 
		
		As before (see the proof of Proposition \ref{P9}), using the nonlinear regularity theory, we infer that $\overline{u} \in C_+ \setminus \{0\}$.
		
		In fact, we have
		\begin{align*}
			&	\mbox{div}\, a(\nabla  \overline{u}(z)) \leq \|\xi\|_\infty  \overline{u}(z)^{p-1} \quad  \mbox{for a.a.  $z \in \Omega$ (see hypothesis $H(\xi)$),} \\ \Rightarrow \quad & \overline{u} \in D_+ \mbox{ (see Pucci-Serrin \cite{Ref23}, pp. 111, 120)}.
		\end{align*}	
		
		Since $V(\cdot)$ is strictly monotone (see hypothesis $H_0$), the solution $\overline{u} \in C_+ \setminus \{0\}$ is unique. Using Proposition 7 of Papageorgiou-R\v{a}dulescu \cite{Ref16}, we have 
		\begin{equation}
		\label{eq7} \|\overline{u}\|_\infty \leq c_{13} \eta^{\frac{1}{p-1}} \quad \mbox{for some } c_{13} >0.
		\end{equation}
		
		Hypotheses $H(g)$ imply that given $\varepsilon >0$, we can find $c_{14} =c_{14}(\varepsilon)>0$ such that 
		\begin{equation}
		\label{eq8} g(z,x) \leq \varepsilon x^{p-1}+c_{14} x^{p^\ast-1} \quad \mbox{for a.a. $z \in \Omega$
		and
		 all $x \geq 0$}.
		\end{equation}
		
		Combining \eqref{eq7} and \eqref{eq8} we have 
		$$g(z, \overline{u}(z)) \leq \varepsilon  \overline{u}(z)^{p-1}+c_{14}  \overline{u}(z)^{p^\ast-1}  \leq \varepsilon c_{13}\eta + c_{14} c_{13} \eta^{\frac{p^\ast-1}{p-1}} .$$
		
		Since $p < p^\ast$, choosing $\eta \in (0,1)$ and $\varepsilon >0$ small, we can have
		\begin{equation}
		\label{eq9} g(z, \overline{u}(z))< \frac{\eta}{2} \quad \mbox{for a.a. } z \in \Omega.
		\end{equation}
		
		Notice that $0 \leq f(z, \overline{u}(z))\leq c_{15}$ for a.a. $z \in \Omega$ and some $c_{15} >0$ (see hypothesis $H(f) \, (i)$). So, choosing $\lambda >0$ small we can have that 
		\begin{equation}
		\label{eq10} \lambda f(z, \overline{u}(z))< \frac{\eta}{2} \quad \mbox{for a.a. } z \in \Omega.
		\end{equation}
		
		It follows from \eqref{eq9} and \eqref{eq10}  that
		\begin{equation}
		\label{eq11} - \mbox{div}\, a(\nabla  \overline{u}(z)) + \xi(z)  \overline{u}(z)^{p-1}= \eta > \lambda f(z, \overline{u}(z)) + g(z, \overline{u}(z)) \quad  \mbox{for a.a.  }z \in \Omega.
		\end{equation}
		
		We introduce the following truncation of the reaction in problem \eqref{PL}
		\begin{equation}
		\label{eq12} k_\lambda(z,x)=\begin{cases} \lambda f(z, x^+)+ g(z, x^+) & \mbox{if } x \leq \overline{u}(z),\\ \lambda f(z, \overline{u}(z))+g(z, \overline{u}(z)) & \mbox{if } \overline{u}(z) < x.\end{cases}
		\end{equation}
		
		This is a Carath\'eodory function. We set $K_\lambda(z,x)=\int_0^xk_\lambda(z,s)ds$ and consider the $C^1$-functional $\psi_\lambda: W^{1,p}(\Omega) \to \mathbb{R}$ defined by
		\begin{align*}\psi_\lambda(u) & =\int_\Omega \widehat{G} (\nabla u)dz +\frac{1}{p}\int_\Omega \xi(z)|u|^p dz+\frac{1}{p}\int_{\partial \Omega}\beta(z)|u|^p d \sigma-\int_\Omega K_\lambda(z,u)dz \\ & \hskip 9cm \mbox{for all } u \in W^{1,p}(\Omega).\end{align*}
		
		We have 
		\begin{align*}\psi_\lambda(u) &\geq \frac{1}{p} \left[\frac{c_1}{p-1}\|\nabla u\|^p_p +  \int_\Omega \xi(z)|u|^p dz+\int_{\partial \Omega}\beta(z)|u|^p d \sigma \right]-\int_\Omega K_\lambda(z,u)dz \\ & \geq \frac{c_{15}}{p}\|u\|^p-c_{16} \quad \mbox{for some $c_{15},c_{16}>0$ (see Lemmata \ref{Lem5} and \ref{Lem6}),} \\ \Rightarrow \quad & \psi_\lambda(\cdot) \mbox{ is coercive}.\end{align*}
		
		Also, by the Sobolev embedding theorem and the compactness of the trace map, we see that $\psi_\lambda(\cdot)$ is sequentially weakly lower semicontinuous. So, by the Weierstrass-Tonelli theorem, we can find $u_\lambda \in W^{1,p}(\Omega)$ such that
		\begin{equation}
		\label{eq13} \psi_\lambda(u_\lambda)= \inf \{\psi_\lambda(u): u \in W^{1,p}(\Omega)\}.
		\end{equation}
		
		Let $V \subseteq \Omega$ be open with $C^1$-boundary such that $\overline{\mathcal{U}} \subseteq V \subseteq \overline{V} \subseteq \Omega$. If $\delta >0$, we define 
		$$V_\delta= \{z \in V: d(z, \partial V)< \delta \}.$$
		
		We can always choose $\delta >0$ small such that 
		\begin{equation}
		\label{eq14}\mathcal{U} \subseteq \overline{V} \setminus V_\delta.
		\end{equation}
		
		We consider a function $\widehat{h} \in C^1(\overline{\Omega})$ such that
		\begin{equation}
		\label{eq15} 0 \leq \widehat{h}(z) \leq 1 \mbox{ for all $z \in \overline{\Omega}$ and } \widehat{h} \Big|_{\overline{V} \setminus V_\delta} =1, \, \widehat{h} \Big|_{\overline{\Omega} \setminus \overline{V}}=0.
		\end{equation}
		
		Hypothesis $H(a)\, (iv)$ implies that we can find $c_{17}\geq c_\ast$ and $\delta \in (0,\delta_0)$  (see hypothesis $H(f) \, (iii)$) such that 
		\begin{equation*}
		\label{eq16} \widehat{G}(y) \leq c_{17}|y|^q \quad \mbox{for all } |y| \leq \delta.
		\end{equation*}
		
		Since $\overline{u} \in D_+$, we can find $t \in (0,1)$ small such that 
		\begin{equation}
		\label{eq17} t \widehat{h} \in (0,\overline{u}] \quad \mbox{and} \quad 0\leq t \widehat{h}(z) \leq \delta \quad \mbox{for all } z \in \overline{V}.
		\end{equation}
		
		We have
		\begin{align*}
			\psi_\lambda(t \widehat{h})&  =\int_{V_\delta} \widehat{G} (t \nabla \widehat{h})dz +\frac{t^p}{p}\int_\Omega \xi(z)\widehat{h}^p dz+\frac{t^p}{p}\int_{\partial \Omega}\beta(z)\widehat{h}^p d \sigma-
		\\&
			\int_\Omega K_\lambda(z,t \widehat{h})dz \mbox{ (see \eqref{eq15})}
			 \leq t^qc_{17} \int_{V_\delta}|\nabla \widehat{h}|^q dz +
			 \frac{t^p}{p} \int_\Omega \xi(z)\widehat{h}^p dz +
			 \\ &
			 \frac{t^p}{p}\int_{\partial \Omega}\beta(z)\widehat{h}^p d \sigma - \frac{\lambda c_9 t^q}{q} \int_{\mathcal{U}}\widehat{h}^q dz \ \ \
			  \mbox{(see \eqref{eq14}, \eqref{eq17} and  $H(f) \, (iii)$)}\\ & = t^q \left[c_{17}\int_{V_\delta}|\nabla \widehat{h}|^q dz - \frac{\lambda c_9}{q} \int_{\mathcal{U}}\widehat{h}^q dz\right] + \frac{t^p}{p} \left[\int_\Omega \xi(z)\widehat{h}^p dz + \int_{\partial \Omega}\beta(z)\widehat{h}^p d \sigma \right]. \end{align*}
		
		We see that if we choose $\delta >0$ small (so that $|V_\delta|_N$ is small) and $t \in (0,1)$ small, too, since $q<p$, we will have
		\begin{align*}
			& \psi_\lambda(t \widehat{h}) <0, \\ \Rightarrow \quad & \psi_\lambda(u_\lambda) <0 = \psi_\lambda(0) \quad \mbox{(see \eqref{eq13})},\\ \Rightarrow \quad & u_\lambda \neq 0.
		\end{align*}
		
		From \eqref{eq13} we have \begin{align}
			&\nonumber \psi'_\lambda(u_\lambda)=0, \\ \label{eq18}\Rightarrow \quad & \langle A(u_\lambda),h \rangle + \int_\Omega \xi(z)|u_\lambda|^{p-2}u_\lambda hdz +\int_{\partial \Omega} \beta(z)|u_\lambda|^{p-2}u_\lambda hd \sigma  =\int_\Omega k_\lambda(z,u_\lambda)hdz \\ \nonumber & \hskip 10cm \mbox{for all $h \in W^{1,p}(\Omega)$.}
		\end{align}
		
		In \eqref{eq18} we choose $h=-u_\lambda^- \in W^{1,p}(\Omega)$. Then
		\begin{align*}
			&\frac{c_1}{p-1}\|\nabla u_\lambda^-\|^p_p +\int_\Omega \xi(z)(u_\lambda^-)^pdz +\int_{\partial \Omega} \beta(z)(u_\lambda^-)^p d \sigma = 0,\\ \Rightarrow \quad & c_{18}\|u_\lambda^-\|^p\leq 0 \quad \mbox{for some $c_{18}>0$ (see Lemmata \ref{Lem5} and \ref{Lem6})},\\ \Rightarrow \quad & u_\lambda \geq 0, \, u_\lambda \neq 0.
		\end{align*}
		
		Also, if in \eqref{eq18} we choose $h=(u_\lambda-\overline{u})^+ \in W^{1,p}(\Omega)$, then
		\begin{align*}
			&\langle A(u_\lambda),(u_\lambda-\overline{u})^+ \rangle + \int_\Omega \xi(z)u_\lambda^{p-1}(u_\lambda-\overline{u})^+dz +\int_{\partial \Omega} \beta(z)u_\lambda^{p-1}(u_\lambda-\overline{u})^+d \sigma \\& =\int_\Omega [ \lambda f(z, \overline{u})+ g(z, \overline{u})](u_\lambda-\overline{u})^+dz\\ & \leq \eta \int_\Omega (u_\lambda-\overline{u})^+dz  \quad \mbox{(see \eqref{eq11})}
	\\ &
		= \langle A(\overline{u}),(u_\lambda-\overline{u})^+ \rangle + \int_\Omega \xi(z)\overline{u}^{p-1}(u_\lambda-\overline{u})^+dz +\int_{\partial \Omega} \beta(z)\overline{u}^{p-1}(u_\lambda-\overline{u})^+d \sigma \\ \Rightarrow \quad & 
			\langle A(u_\lambda) -A(\overline{u}),(u_\lambda-\overline{u})^+ \rangle + \int_\Omega \xi(z)(u_\lambda^{p-1} -\overline{u}^{p-1})(u_\lambda-\overline{u})^+dz\\ & \quad  +\int_{\partial \Omega} \beta(z)(u_\lambda^{p-1} -\overline{u}^{p-1})(u_\lambda-\overline{u})^+d \sigma \leq 0\\ \Rightarrow \quad & u_\lambda \leq \overline{u} \quad \mbox{(see Lemmata \ref{Lem5}, \ref{Lem6})}.
		\end{align*}
		
		So, we have proved that
		\begin{equation}
		\label{eq19} u_\lambda \in [0, \overline{u}], \, u_\lambda \neq 0.
		\end{equation}
		
		On account of \eqref{eq12} and \eqref{eq19}, equation \eqref{eq18} becomes
		\begin{align*} &\langle A(u_\lambda),h \rangle + \int_\Omega \xi(z)u_\lambda^{p-1} hdz +\int_{\partial \Omega} \beta(z)u_\lambda^{p-1} hd \sigma \\ & =\int_\Omega [\lambda f(z,u_\lambda)+g(z,u_\lambda)]hdz \quad \mbox{for all $h \in W^{1,p}(\Omega)$,}\\ \Rightarrow \quad & u_\lambda \in S(\lambda) \subseteq D_+ \quad \mbox{(see Proposition \ref{P9}) and so $\lambda \in \mathcal{L} \neq \emptyset$}.
		\end{align*}
	\end{proof}	
	
	In the next proposition, we prove a structural property of $\mathcal{L}$, namely we show that $\mathcal{L}$ is an interval.
	
	\begin{prop}
		\label{P11} If hypotheses $\widetilde{H}$ hold, $\lambda \in \mathcal{L}$ and $0 < \mu < \lambda$, then $\mu \in \mathcal{L}$.
	\end{prop}
	
	\begin{proof}
		Since $\lambda \in \mathcal{L}$, there is $u_\lambda \in S(\lambda) \subseteq D_+$ (see Proposition \ref{P9}). We consider the following truncation of the reaction in problem $(P_\mu)$
		\begin{equation}
		\label{eq20} k_\mu(z,x)=\begin{cases} \mu f(z, x^+)+ g(z, x^+) & \mbox{if } x \leq u_\lambda(z),\\ \mu f(z, u_\lambda(z))+g(z,  u_\lambda(z)) & \mbox{if }  u_\lambda(z) < x.\end{cases}
		\end{equation}
		
		This is a Carath\'eodory function. We set $K_\mu(z,x)=\int_0^xk_\mu(z,s)ds$ and consider the $C^1$-functional $\widehat{\varphi}_\mu: W^{1,p}(\Omega) \to \mathbb{R}$ defined by
		\begin{align*}\widehat{\varphi}_\mu(u) & =\int_\Omega \widehat{G} (\nabla u)dz +\frac{1}{p}\int_\Omega \xi(z)|u|^p dz+\frac{1}{p}\int_{\partial \Omega}\beta(z)|u|^p d \sigma-\int_\Omega K_\mu(z,u)dz \\ & \hskip 9cm \mbox{for all } u \in W^{1,p}(\Omega).\end{align*}
		
		As before we have that
		\begin{align*} & \widehat{\varphi}_\mu(\cdot) \mbox{ is coercive (see \eqref{eq20})},\\ & \widehat{\varphi}_\mu(\cdot) \mbox{ is sequentially weakly lower semicontinuous.} \end{align*}
		
		So, we can find $u_\mu \in W^{1,p}(\Omega)$ such that
		\begin{equation}
		\label{eq21}\widehat{\varphi}_\mu(u_\mu)= \inf \{\widehat{\varphi}_\mu(u): u \in W^{1,p}(\Omega)\}.
		\end{equation}
		
		Reasoning as in the proof of Proposition \ref{P10}, using the cut-off function $\widehat{h}$, we show that 
		\begin{align*}
			& \widehat{\varphi}_\mu(u_\mu) <0=\widehat{\varphi}_\mu(0), \\ \Rightarrow \quad & u_\mu \neq 0.
		\end{align*}
		
		From \eqref{eq21} we have \begin{align}
			&\nonumber\widehat{\varphi}^\prime_\mu(u_\mu)=0, \\ \label{eq22}\Rightarrow \quad & \langle A(u_\mu),h \rangle + \int_\Omega \xi(z)|u_\mu|^{p-2}u_\mu hdz +\int_{\partial \Omega} \beta(z)|u_\mu|^{p-2}u_\mu hd \sigma  =\int_\Omega k_\mu(z,u_\mu)hdz \\ \nonumber & \hskip 10cm \mbox{for all $h \in W^{1,p}(\Omega)$.}
		\end{align}
		
		In \eqref{eq22} we first choose $h=-u_\mu^- \in W^{1,p}(\Omega)$ and infer that
		$$u_\mu \geq 0, \, u_\mu \neq 0.$$
		
		Next, in \eqref{eq22} we choose $h=(u_\mu - u_\lambda)^+ \in W^{1,p}(\Omega)$. We have
		\begin{align*}
			&\langle A(u_\mu),(u_\mu-u_\lambda)^+ \rangle + \int_\Omega \xi(z)u_\mu^{p-1}(u_\mu-u_\lambda)^+dz +\int_{\partial \Omega} \beta(z)u_\mu^{p-1}(u_\mu-u_\lambda)^+d \sigma \\& =\int_\Omega [ \mu f(z, u_\lambda)+ g(z, u_\lambda)](u_\mu-u_\lambda)^+dz \quad \mbox{(see \eqref{eq20})} \\& \leq\int_\Omega [ \lambda f(z, u_\lambda)+ g(z, u_\lambda)](u_\mu-u_\lambda)^+dz \quad \mbox{(since $\lambda > \mu$)}\\ & =\langle A(u_\lambda),(u_\mu-u_\lambda)^+ \rangle + \int_\Omega \xi(z)u_\lambda^{p-1}(u_\mu-u_\lambda)^+dz +\int_{\partial \Omega} \beta(z)u_\lambda^{p-1}(u_\mu-u_\lambda)^+d \sigma \\ & \hskip 10cm \mbox{(since $u_\lambda \in S(\lambda)$),}\\ \Rightarrow \quad & u_\mu \leq u_\lambda.
		\end{align*}
		
		So, we have proved that
		\begin{align*}
			& u_\mu \in [0, u_\lambda], \, u_\mu \neq 0, \\ \Rightarrow \quad & u_\mu \in S(\mu) \subseteq D_+ \quad \mbox{(see \eqref{eq20}, \eqref{eq22} and Proposition \ref{P9})},\\ \Rightarrow \quad & \mu \in \mathcal{L}.
		\end{align*}
	\end{proof}	
	
	This proposition shows that $\mathcal{L}$ is an interval. An interesting byproduct of the above proof is the following corollary.
	
	\begin{cor}
		\label{C12} If hypotheses $\widetilde{H}$ hold, $\lambda \in \mathcal{L}$, $u_\lambda \in S(\lambda) \subseteq D_+$ and $0 < \mu < \lambda$, then $\mu \in \mathcal{L}$ and we can find $u_\mu \in S(\mu) \subseteq D_+$ such that $u_\lambda - u_\mu \in C_+ \setminus \{0\}$.
	\end{cor}
	
	We can improve the conclusion of this corollary.
	
	\begin{prop}
		\label{P13} If hypotheses $\widetilde{H}$ hold, $\lambda \in \mathcal{L}$, $u_\lambda \in S(\lambda) \subseteq D_+$ and $0 < \mu < \lambda$, then $\mu \in \mathcal{L}$ and we can find $u_\mu \in S(\mu) \subseteq D_+$ such that $u_\lambda - u_\mu \in {\rm int}\, \widehat{C}_+ $.
	\end{prop}
	
	\begin{proof}
		From Corollary \ref{C12}, we already know that $\mu \in \mathcal{L}$ and we can find $u_\mu \in S(\mu) \subseteq D_+$ such that 
		\begin{equation} \label{eq23} u_\lambda - u_\mu \in C_+ \setminus \{0\}. \end{equation}
		
		Let $\rho=\|u_\lambda\|_\infty$, $B=[\mu,\lambda]$ and let $\widehat{\xi}^B_\rho>0$ as postulated by hypothesis $\widehat{H}_0 \, (ii)$. We have 
		\begin{align} \nonumber &
			- \mbox{div}\, a(\nabla  u_\mu(z)) + [\xi(z) + \widehat{\xi}^B_\rho]  u_\mu(z)^{p-1} \\ \nonumber & = \mu f(z,  u_\mu(z)) + g(z,  u_\mu(z))+ \widehat{\xi}^B_\rho   u_\mu(z)^{p-1} \\ \nonumber & \leq \mu f(z,  u_\lambda(z)) + g(z,  u_\lambda(z))+ \widehat{\xi}^B_\rho   u_\lambda(z)^{p-1} \quad \mbox{(see \eqref{eq23} and hypothesis $\widehat{H}_0 \, (ii)$)}\\ \label{eq24} & = \lambda f(z,  u_\lambda(z)) + g(z,  u_\lambda(z))+ \widehat{\xi}^B_\rho   u_\lambda(z)^{p-1} +[\mu-\lambda]f(z,  u_\lambda(z)) \quad \mbox{for a.a. }z \in \Omega.
		\end{align}
		
		Recall that $u_\lambda \in D_+$. Therefore $s_\lambda = \min_{\overline{\Omega}} u_\lambda >0$. Then using hypothesis $H(f) \, (iii)$, we have 
		\begin{equation}
		\label{eq25} f(z,  u_\lambda(z)) \geq \eta_{s_\lambda}
		>0 \quad \mbox{for a.a. }z \in \Omega.\end{equation}
		
		Using \eqref{eq25} in \eqref{eq24} and recalling that $\mu < \lambda$, we obtain
		\begin{align*} &
			- \mbox{div}\, a(\nabla  u_\mu(z)) + [\xi(z) + \widehat{\xi}^B_\rho]  u_\mu(z)^{p-1} \\  & \leq \lambda f(z,  u_\lambda(z)) + g(z,  u_\lambda(z))+ \widehat{\xi}^B_\rho   u_\lambda(z)^{p-1} +[\mu-\lambda] \eta_{s_\lambda}\\ & < 	- \mbox{div}\, a(\nabla  u_\lambda(z)) + [\xi(z) + \widehat{\xi}^B_\rho]  u_\lambda(z)^{p-1} \quad \mbox{for a.a. }z \in \Omega, \\ \Rightarrow \quad & u_\lambda - u_\mu \in {\rm int \ } \widehat{C}_+ \quad \mbox{(see Proposition \ref{P8})}.
		\end{align*}
	\end{proof}
	
	We set $\lambda^\ast = \sup \mathcal{L}$.
	
	\begin{prop}
		\label{P14} If hypotheses $\widetilde{H}$ hold, then $\lambda^\ast < +\infty$.
	\end{prop}
	
	\begin{proof}
		Let $\mu > \|\xi\|_\infty$ (see hypothesis $H(\xi)$). We claim that we can find $\widehat{\lambda}>0$ such that
		\begin{equation}
		\label{eq26} \widehat{\lambda}f(z,x)+g(z,x) \geq \mu x^{p-1} \quad \mbox{for a.a. $z \in \overline{\mathcal{U}}$
		and
		 all $x \geq 0$.}
		\end{equation}
		
		To this end, notice that for any $\lambda >0$ on account of hypothesis $H(f) \, (iii)$ we have 
		\begin{equation}
		\label{eq27} \lambda f(z,x) \geq  \mu x^{p-1} \quad \mbox{for a.a. $z \in \overline{\mathcal{U}}$, all $0 \leq x \leq \widehat{\delta} \leq \delta_0$ (recall $q<p$ and $\delta_0 \leq 1$).}
		\end{equation}
		
		Also, hypothesis $H(g) \, (ii)$ implies that we can find $M_1>0$ such that 
		\begin{equation}
		\label{eq28}   g(z,x) \geq  \mu x^{p-1} \quad \mbox{for a.a. $z \in \Omega$
		and
		 all $x \geq M_1$}.
		\end{equation}
		
		According to  hypothesis $H(f) \, (ii)$, we have 
		\begin{equation}
		\label{eq29} \lambda f(z,x) \geq  \lambda \eta_{\widehat{\delta}} \quad \mbox{for a.a. $z \in \Omega$
		and
		 all $x \geq \widehat{\delta}$}.
		\end{equation}
		
		Choose $\widehat{\lambda} >0$ such that
		\begin{equation}
		\label{eq30} \lambda \eta_{\widehat{\delta}} \geq \mu M_1^{p-1}.
		\end{equation}
		
		Then from \eqref{eq27}, \eqref{eq28}, \eqref{eq29}, \eqref{eq30} and since $f,g \geq 0$, we conclude that \eqref{eq26} is true.
		
		Now let $\lambda > \widehat{\lambda}$ and assume that $\lambda \in \mathcal{L}$. Then we can find $u_\lambda \in S(\lambda) \subseteq D_+$.
		We set $m_\lambda = \min_{\overline{\mathcal{U}}} u_\lambda >0$. For $\delta >0$ let $m_\lambda^\delta =m_\lambda + \delta$. We set $\rho = \|u_\lambda \|_\infty$, $B= \{\lambda\}$ and consider $\widehat{\xi}^B_\rho >0$ as postulated by hypothesis $\widehat{H}_0 \, (ii)$. We have 
		\begin{align} &
			\nonumber	- \mbox{div}\, a(\nabla  m_\lambda^\delta) + [\xi(z) + \widehat{\xi}^B_\rho]  (m_\lambda^\delta)^{p-1} \\ \nonumber & =  [\xi(z) + \widehat{\xi}^B_\rho]  (m_\lambda^\delta)^{p-1}\\ \nonumber & \leq [\xi(z) + \widehat{\xi}^B_\rho]  m_\lambda^{p-1} + \chi(\delta) \quad \mbox{ with $\chi(\delta) \to 0^+$ as $\delta \to 0^+$}\\ \nonumber & <   [\mu + \widehat{\xi}^B_\rho]  m_\lambda^{p-1} + \chi(\delta) \quad \mbox{ (recall that $\mu >\|\xi\|_\infty$)}\\  & \leq \widehat{\lambda} f(z,  m_\lambda ) + g(z,  m_\lambda )+ \widehat{\xi}^B_\rho   m_\lambda ^{p-1} +\chi(\delta)
			\quad \mbox{(see \eqref{eq26})} \nonumber  \\ \nonumber & = \lambda f(z,  m_\lambda) + g(z,  m_\lambda)+ \widehat{\xi}^B_\rho   m_\lambda^{p-1} +[\widehat{\lambda}-\lambda] f(z,m_\lambda)+\chi(\delta)\\ & \leq \lambda f(z,  m_\lambda) + g(z,  m_\lambda)+ \widehat{\xi}^B_\rho   m_\lambda^{p-1} +[\widehat{\lambda}-\lambda] \eta_{m_\lambda}+\chi(\delta) \label{eq31}\\ \nonumber & \hskip 4cm \quad \mbox{(see hypothesis $H(f) \, (ii)$ and recall that $\widehat{\lambda}<\lambda$)}.
		\end{align}
		
		Since $\chi(\delta) \to 0^+$ as $\delta \to 0^+$, for $\delta >0$ small we have
		\begin{equation}
		\label{eq32} \chi(\delta)   < [\lambda-\widehat{\lambda}] \eta_{m_\lambda} \quad \mbox{(recall
		that
		 $\widehat{\lambda}<\lambda$)}.
		\end{equation}
		
		Using \eqref{eq32} in \eqref{eq31}, we see that for $\delta >0$ small, we have
		\begin{align*} &
			- \mbox{div}\, a(\nabla  m_\lambda^\delta) + [\xi(z) + \widehat{\xi}^B_\rho]  (m_\lambda^\delta)^{p-1} \\ \nonumber & < \lambda f(z,  u_\lambda(z)) + g(z,  u_\lambda(z))+ \widehat{\xi}^B_\rho   u_\lambda(z)^{p-1}\\ & \hskip 4cm \mbox{(see hypothesis $\widehat{H}_0 \, (ii)$ and recall
			that 
			 $m_\lambda = \min_{\overline{\mathcal{U}}} u_\lambda$)}\\ & = 	- \mbox{div}\, a(\nabla u_\lambda(z)) + [\xi(z) + \widehat{\xi}^B_\rho]  u_\lambda(z)^{p-1} \quad \mbox{for a.a. }z \in \overline{\mathcal{U}}\\ \Rightarrow \quad & u_\lambda - m_\lambda^\delta \in {\rm int \ } \widehat{C}_+(\overline{\mathcal{U}}) \quad \mbox{for $\delta >0$ small (see Proposition \ref{P8}).}
		\end{align*}
		
		This contradicts the fact that $m_\lambda = \min_{\overline{\mathcal{U}}} u_\lambda$. It follows that $\lambda \not \in \mathcal{L}$ and so we conclude that $\lambda^\ast \leq \widehat{\lambda} < +\infty$.
	\end{proof}	

\begin{prop}\label{prop15}
	If hypotheses $\widetilde{H}$ hold and $0< \lambda <\lambda^*$, then problem \eqref{PL} has at least two positive solutions $u_0, \widehat{u} \in D_+$, $u_0 \neq \widehat{u}$.
\end{prop}
\begin{proof}
	Let $0<\lambda_1 <\lambda <\lambda_2 < \lambda^*$. We know that $\lambda_1, \lambda_2 \in \mathcal{L}$. According to Proposition \ref{P13}, we can find $u_{\lambda_2}
	\in S(\lambda_2) \subseteq D_+$ and $u_{\lambda_1}
	\in S(\lambda_1) \subseteq D_+$ such that 
	$$u_{\lambda_2} - u_{\lambda_1} \in {\rm int}\, C_+.$$
	We consider the following truncation of the reaction in problem \eqref{PL} 

	\begin{equation}\label{eq33}
		\widehat{j}_\lambda(z,x)=\begin{cases}
			\lambda f(z,u_{\lambda_1}(z))+g(z,u_{\lambda_1}(z)) & \mbox{if $x <u_{ \lambda_1}(z)$},\\
			\lambda f(z,x)+g(z,x) &  \mbox{if $u_{ \lambda_1}(z) \leq x \leq u_{ \lambda_2}(z)$},\\
			\lambda f(z,u_{\lambda_2}(z))+g(z,u_{\lambda_2}(z)) &  \mbox{if $u_{\lambda_2}(z) <x$}.
		\end{cases}
	\end{equation} 
	This is a Carath\'{e}odory function. We set $\widehat{J}_\lambda(z,x)= \int_0^x \widehat{j}_\lambda(z,s)ds$ and consider the $C^1$-functional $\widehat{\tau}_\lambda:W^{1,p}(\Omega) \to \mathbb{R}$ defined by
	$$\widehat{\tau}_\lambda(u)= \int_\Omega \widehat{G}(\nabla u) dz +\frac{1}{p}\int_\Omega\xi(z) |u|^pdz + \frac{1}{p}\int_{\partial \Omega}\beta(z)|u|^pd\sigma -\int_\Omega\widehat{J}_\lambda(z,u)dz$$
	for all $u \in W^{1,p}(\Omega)$. 
	Evidently,  $\widehat{\tau}_\lambda(\cdot)$ is coercive (see \eqref{eq33} and Lemmata \ref{Lem5} and \ref{Lem6}) and sequentially weakly lower semicontinuous. So, we can find $u_0 \in W^{1,p}(\Omega)$ such that 
	\begin{align}\label{eq34}
		&\widehat{\tau}_\lambda(u_0) = \inf \{\widehat{\tau}_\lambda(u):\ u \in W^{1,p}(\Omega)\},\\
		\nonumber \Rightarrow \quad & \widehat{\tau}_\lambda^\prime(u_0) =0,\\
		\nonumber \Rightarrow \quad & \langle A(u_0),h \rangle + \int_\Omega\xi(z) |u_0|^{p-2}u_0hdz + \int_{\partial \Omega}\beta(z)|u_0|^{p-2}u_0hd\sigma =\int_\Omega\widehat{j}_\lambda(z,u_0)hdz\\ \nonumber
		& \hspace{8cm} \mbox{ for all $h \in W^{1,p}(\Omega)$}.
	\end{align}
	Choosing $h=(u_{\lambda_1} -u_0)^+ \in W^{1,p}(\Omega)$ and  $h=(u_0 - u_{\lambda_2})^+ \in W^{1,p}(\Omega)$ and reasoning as before, we obtain that
	\begin{align*}
		& u_0 \in [u_{\lambda_1}, u_{\lambda_2}],\\
		\Rightarrow \quad & u_0 \in S(\lambda) \subseteq D_+ \quad \mbox{ (see \eqref{eq33})}.
	\end{align*}
	In fact, using Proposition \ref{P8} (the strong comparison principle) as in the proof of  Proposition \ref{P13}, we obtain
	\begin{equation}\label{eq35}
		u_0 \in {\rm int}_{C^1(\overline{\Omega})}[u_{\lambda_1}, u_{\lambda_2}].
	\end{equation}
	Consider the following Carath\'{e}odory function
	\begin{equation}\label{eq36}
		j_\lambda(z,x)=\begin{cases}
			\lambda f(z,u_{\lambda_1}(z))+g(z,u_{\lambda_1}(z)) & \mbox{if $x \leq u_{\lambda_1}(z)$},\\
			\lambda f(z,x)+g(z,x) &  \mbox{if $u_{ \lambda_1}(z) < x$}.
		\end{cases}
	\end{equation} 
	We set $J_\lambda(z,x)= \int_0^x j_\lambda(z,s)ds$ and consider the $C^1$-functional $\tau_\lambda:W^{1,p}(\Omega) \to \mathbb{R}$ defined by
	$$\tau_\lambda(u)= \int_\Omega \widehat{G}(\nabla u) dz +\frac{1}{p}\int_\Omega\xi(z) |u|^pdz + \frac{1}{p}\int_{\partial \Omega}\beta(z)|u|^pd\sigma -\int_\Omega J_\lambda(z,u)dz$$
	for all $u \in W^{1,p}(\Omega)$. 
	
	From \eqref{eq33} and \eqref{eq36} it is clear that 
	$$\widehat{\tau}_\lambda \Big|_{[u_{\lambda_1}, u_{\lambda_2}]}= \tau_\lambda \Big|_{[u_{\lambda_1}, u_{\lambda_2}]}.$$
	From \eqref{eq34} and \eqref{eq35} we infer that 
	\begin{align}\label{eq37}
		\nonumber & u_0 \mbox{ is a local $C^1(\overline{\Omega})$-minimizer of $\tau_\lambda$},\\
		\Rightarrow \quad &  u_0 \mbox{ is a local $W^{1,p}(\Omega)$-minimizer of $\tau_\lambda$ (see Proposition \ref{P7})}.
	\end{align}
	Using \eqref{eq36}, we can easily check that 
	\begin{equation}\label{eq38}
		K_{\tau_\lambda} \subseteq [u_{\lambda_1} ) \cap D_+.
	\end{equation}
	So, we may assume that $K_{\tau_\lambda}$ is finite (otherwise we already have an infinity of positive solutions in $D_+$, see \eqref{eq36}). Then this property of $K_{\tau_\lambda}$ and \eqref{eq37} imply that we can find $\rho \in (0,1)$ small such that 
	\begin{equation}\label{eq39}
		\tau_\lambda(u_0) < \inf \{\tau_\lambda(u): \, \|u-u_0\|=\rho\}=m_\rho^\lambda
	\end{equation}
	(see Aizicovici-Papageorgiou-Staicu \cite{Ref1}, proof of Proposition 29).
	
	Given $u \in D_+$, on account of hypothesis $H(g)\,(ii)$, we have
	\begin{equation}\label{eq40}
		\tau_\lambda(tu) \to - \infty\quad \mbox{ as } \quad t\to + \infty.
	\end{equation}
	
	\medskip
	\noindent \underline{Claim}: $\tau_\lambda$ satisfies the $C$-condition.
	
	Let $\{ u_{n} \}_{n \in \mathbb{N}} \subseteq W^{1, p} (\Omega)$ be a sequence such that 
	\begin{equation} \label{eq41}
		|\tau_\lambda(u_n)| \leq M_2 \quad \mbox{for some } M_2>0
		 \mbox{ and  all } n \in \mathbb{N},\end{equation}
	\begin{equation} \label{eq42}
		(1 + \|u_{n}\|) \tau_\lambda^{\prime} (u_n) \rightarrow 0 \mbox{  in } W^{1, p} (\Omega) ^* \mbox{ as } n \to +\infty. 
	\end{equation}
	From \eqref{eq42} we have 
	\begin{align} \label{eq43} 
		\nonumber & |\langle \tau_\lambda^{\prime}(u_n), h \rangle |  \leq\frac{\varepsilon_n \|h\|}{1 + \|u_n\|}\quad \mbox{for all $h \in W^{1,p} (\Omega)$, with $\varepsilon_n \rightarrow 0^+$,}\\
		\nonumber \Rightarrow & \Big| \langle A(u_n), h \rangle + \int_\Omega \xi(z) |u_n|^{p - 2} u_n  h dz + \int_{\partial \Omega}\beta(z) |u_n|^{p-2} u_n  h d \sigma 
		-  \int_\Omega j_\lambda(z, u_n)  h dz \Big| \\ &
		\leq\frac{\varepsilon_n \|h\|}{1 + \|u_n\|}, \quad \mbox{for all $h \in W^{1,p} (\Omega)$, 
		 $n \in \mathbb{N}$.}
	\end{align}

	In \eqref{eq43} we choose $h = -u_n^- \in W^{1,p} (\Omega)$. Using Lemma \ref{Lem2}, we have  
	\begin{align}\label{eq44}
		& \nonumber \frac{c_1}{p-1}\|\nabla u_n^- \|^p_p + \int_\Omega \xi(z) (u_n^-)^{p} \, dz + \int_{\partial \Omega}\beta(z) (u_n^-)^{p}d \sigma  \leq c_{19} \|u_n^-\|\\ 
		& \nonumber \hspace{4cm} \mbox{for some $c_{19} >0$
		and
		  all } n \in \mathbb{N}
		\mbox{ (see \eqref{eq36})},\\
		\nonumber \Rightarrow \quad & \|u_n^-\|^{p-1} \leq c_{20} \mbox{ for some $c_{20}>0$
		and
		 all } n \in \mathbb{N} \mbox{ (see Lemmata \ref{Lem5} and \ref{Lem6})},\\
		\Rightarrow \quad & \{u_n^-\}_{n \in \mathbb{N}} \subseteq W^{1,p}(\Omega) \mbox{ is bounded}.
	\end{align}
	Using \eqref{eq44} in \eqref{eq41}, we obtain
	\begin{equation}\label{eq45}
		\int_\Omega p \widehat{G}(\nabla u_n)dz + \int_\Omega \xi(z) (u_n^+)^{p}  dz + \int_{\partial \Omega}\beta(z) (u_n^+)^{p}  d \sigma  -  \int_\Omega p[\lambda F(z, u_n^+)+ G(z, u_n^+)] dz \leq M_3,
	\end{equation} 
	for some $M_3 >0$
	 and all $n \in \mathbb{N}$ (see \eqref{eq36}). 
	
	On the other hand, if in \eqref{eq43} we choose $h = u_n^+ \in W^{1,p}(\Omega)$, then
	\begin{align}\label{eq46}
	&	-\int_\Omega (a(\nabla u_n^+),\nabla u_n^+)_{\mathbb{R}^N}dz - \int_\Omega \xi(z) (u_n^+)^{p}  dz - \int_{\partial \Omega}\beta(z) (u_n^+)^{p}  d \sigma \\ \nonumber &   +  \int_\Omega [\lambda f(z, u_n^+)+g(z, u_n^+)] u_n^+  dz \leq c_{21},   \mbox{for some $c_{21}>0$
	and
	 all $n \in \mathbb{N}$ (see \eqref{eq36}).}
	\end{align} 
	 We add \eqref{eq45} and \eqref{eq46} and using hypothesis $H(a)\, (iv)$, we obtain
	\begin{equation}
		\label{eq47} \int_\Omega e_\lambda(z,u_n^+)dz \leq M_4, \quad \mbox{for some } M_4>0 \ \ 
		 \mbox{and all } n \in \mathbb{N}.
	\end{equation}
	
	We will show that   $\{ u^+_{n} \}_{n \in \mathbb{N}} \subseteq W^{1, p} (\Omega)$ is bounded. Arguing by contradiction, suppose that 
	$\| u_{n}^+  \| \rightarrow + \infty$  as $n \rightarrow + \infty$.
	
	We set $y_n = \dfrac{u_n^+}{\| u_{n}^+  \| }$,   $n \in \mathbb{N}$. Then  $\| y_n \| = 1$, $y_n \geq 0$ for all $n \in \mathbb{N}$. We may assume that 
	\begin{equation} \label{eq48} y_n \xrightarrow{w} y \mbox{ in } W^{1, p} (\Omega) \mbox{ and } y_n \rightarrow y \mbox{ in } L^p(\Omega), \mbox{ and in } L^p(\partial \Omega), \, y \geq 0. 
	\end{equation}
	First, assume that $y \neq 0$ and let $\Omega_+=\{z \in \Omega : y(z) >0\}$. We have $|\Omega_+|_N>0$ (recall that $y \geq 0$, see \eqref{eq48}).  Then
	\begin{equation}\label{eq49} u_{n}^+(z)  \rightarrow + \infty \quad \mbox{for all } z \in \Omega_+.
	\end{equation}
	Hypotheses $H(f)\,(ii)$ and $H(g)\,(ii)$ imply that 
	\begin{equation}\label{eq50} 
		\lim_{x \to + \infty}\dfrac{F(z, x)}{x^{p}}=0 \mbox{ and } \to \lim_{x \to + \infty}\dfrac{G(z, x)}{x^{p}}=+\infty
		\mbox{ uniformly for a.a. } z \in \Omega.
	\end{equation}
	Then \eqref{eq49}, \eqref{eq50} imply that 
	$$\dfrac{F(z, u_n^+(z))}{\|u_n^+\|^{p}}\to 0 \quad \mbox{for a.a. } z \in \Omega_+,$$
	$$\dfrac{G(z, u_n^+(z))}{\|u_n^+\|^{p}}\to + \infty \quad \mbox{for a.a. } z \in \Omega_+.$$
	Using Fatou's lemma, we have
	\begin{align}\label{eq51} 
		\nonumber
		&\int_{\Omega_+}\dfrac{\lambda F(z, u_n^+)+G(z, u_n^+)}{\|u_n^+\|^{p}}dz \to +\infty \mbox{ as } n \to + \infty\\
		\Rightarrow \quad & \int_{\Omega}\dfrac{\lambda F(z, u_n^+)+G(z, u_n^+)}{\|u_n^+\|^{p}}dz \to +\infty \mbox{ as } n \to + \infty \mbox{ (recall $F,G \geq 0$)}.
	\end{align}
	Recall that from \eqref{eq41} and \eqref{eq44}, we have
	\begin{align}\label{eq52}
		\nonumber & \left| \int_\Omega p \widehat{G}(\nabla u_n^+)dz + \int_\Omega \xi(z) (u_n^+)^{p}  dz + \int_{\partial \Omega}\beta(z) (u_n^+)^{p}  d \sigma \right.\\
		& \hskip 1cm \left. -  \int_\Omega p[\lambda F(z, u_n^+)+ G(z, u_n^+)] dz\right| \leq M_5,  \mbox{for some $M_5>0$
		and  all $n \in \mathbb{N}$} \nonumber \\
		\nonumber\Rightarrow \ & \int_\Omega p[\lambda F(z, u_n^+)+ G(z, u_n^+)]dz \leq \int_\Omega p \widehat{G}(\nabla u_n^+)dz + \int_\Omega \xi(z) (u_n^+)^{p}  dz\\
		 & \hspace{6cm} + \int_{\partial \Omega}\beta(z) (u_n^+)^{p}  d \sigma + M_4
		\mbox{ for  all $n \in \mathbb{N}$} \nonumber \\
		\nonumber \Rightarrow \ & \int_\Omega \frac{p[\lambda F(z, u_n^+)+ G(z, u_n^+)]}{\|u_n^+\|^p} dz \leq \frac{1}{\|u_n^+\|^p} \int_\Omega p \widehat{G}(\nabla u_n^+)dz
		+ \int_\Omega \xi(z) y_n^{p}  dz\\
		& \hspace{4cm}  + \int_{\partial \Omega}\beta(z) y_n^{p}  d \sigma +\frac{M_4}{\|u_n^+\|^p} \mbox{ for all $n \in \mathbb{N}$}. 
	\end{align}
	
	Corollary \ref{cor3} and hypothesis $H(a)\,(iv)$ imply that 
	
	$$\widehat{G}(y) \leq c_{22} (|y|^q+|y|^p) \mbox{ for some $c_{22}>0$
	and  all $y \in \mathbb{R}^N$.}$$
	Therefore we have
	\begin{equation}\label{eq53}
		\frac{1}{\|u_n^+\|^p} \int _\Omega p \widehat{G}(\nabla u_n^+) dz \leq \frac{p\,c_{22}}{\|u_n^+\|^{p-q}} \|\nabla y_n\|^q_q+ p\,c_{22} \|\nabla y_n\|^p_p \leq M_6
	\end{equation}
	for some $M_6>0$, all $n \in \mathbb{N}$ (recall $p>q$).
	Returning to \eqref{eq52} and using \eqref{eq53}, we obtain 
	\begin{equation}\label{eq54}
		\int_\Omega \frac{p[\lambda F(z, u_n^+)+ G(z, u_n^+)]}{\|u_n^+\|^p} dz \leq M_7 \quad \mbox{for some $M_7>0$
		and all $n \in \mathbb{N}$}.
	\end{equation}
	Comparing \eqref{eq51} and \eqref{eq54}, we have a contradiction.
	
	Next, we assume that $y=0$. We introduce the $C^1$-functional
	$\tau_\lambda^*: W^{1,p}(\Omega) \to \mathbb{R}$ defined by
	
	$$\tau_\lambda^*(u)= \frac{c_1}{p(p-1)}\|\nabla u\|^p_p+\frac{1}{p} \int_\Omega\xi(z) |u|^pdz +\frac{1}{p} \int_{ \partial \Omega} \beta(z) |u|^pd\sigma -\int_\Omega J_\lambda(z,u)dz$$ 
	for all $u \in W^{1,p}(\Omega)$.

Let $k>0$ and define
$$v_n=(kp)^{1/p}y_n \in  W^{1,p}(\Omega) \quad \mbox{for all }n \in \mathbb{N}.$$

We have 
	\begin{equation} \label{eq55} v_n \xrightarrow{w} 0 \mbox{ in } W^{1, p} (\Omega) \mbox{ and } v_n \rightarrow 0 \mbox{ in } L^p(\Omega) \mbox{ and in } L^p(\partial \Omega), \mbox{ (see \eqref{eq48} and recall  $y=0$)}. 
\end{equation}

Hypotheses $H(f) \, (i), (ii)$ imply that
\begin{align*}
&0 \leq F(z,x) \leq c_{23}(1+x^{p-1}) \quad \mbox{for a.a. $z \in \Omega$, all $x \geq 0$, 
and
some $c_{23}>0$},\\ \Rightarrow \quad & 	\int_\Omega F(z,v_n) dz \to 0 \quad \mbox{(see \eqref{eq55})}.
\end{align*}	

Let $c_{24}=\sup_{n \in \mathbb{N}} \|v_n \|^{p^\ast}_{p^\ast}$ (see \eqref{eq55}). Hypotheses  $H(g) \, (i), (ii)$ imply that given $\varepsilon >0$, we can find $c_{25}=c_{25}(\varepsilon)>0$ such that
\begin{equation}
\label{eq56}0 \leq G(z,x) \leq \frac{\varepsilon}{2 c_{24}}x^{p^\ast}
+ c_{25} \quad \mbox{for a.a. $z \in \Omega$
and
 all $x \geq 0$.}\end{equation}

Let $E \subseteq \Omega$ be a measurable set with $|E|_N \leq  \dfrac{\varepsilon}{2 c_{25}}$. Then we have
$$\int_E G(z,v_n)dz \leq  \frac{\varepsilon}{2 c_{24}}\|v_n \|^{p^\ast}_{p^\ast} +c_{25}|\Omega|_N \leq \varepsilon \quad \mbox{for all $n \in \mathbb{N}$ (see \eqref{eq56})}.$$

Also, from \eqref{eq54} we see that 
$$\{N_G(v_n)\}_{n \in \mathbb{N}} \subseteq L^1(\Omega) \mbox{ is bounded}.$$

It follows that
\begin{equation}\label{eq57}\{N_G(v_n)\}_{n \in \mathbb{N}} \subseteq L^1(\Omega) \mbox{ is uniformly integrable}\end{equation}
(see Gasi\'nski-Papageorgiou \cite{Ref8}, Problem 1.6, p. 36).

From \eqref{eq55} and by passing to a subsequence if necessary, we can say that
\begin{align*}
&v_n(z) \to 0 \quad \mbox{for a.a. $z \in \Omega$},\\ 	
\Rightarrow \quad &G(z,v_n(z)) \to 0 \quad \mbox{for a.a. $z \in \Omega$},\\
\Rightarrow \quad & \int_ \Omega G(z,v_n)dz \to 0 \quad \mbox{as $n \to +\infty$},\\
\end{align*}	
using Vitali's Theorem (see Gasi\'nski-Papageorgiou \cite{Ref8}, p. 5), we have
\begin{equation*}
\label{eq58} \int_ \Omega [\lambda F(x,v_n)+G(z,v_n)]dz \to 0 \quad \mbox{as $n \to +\infty$}
\end{equation*}

Recall that $\|u_n^+\| \to +\infty$. So, we can find $n_0 \in \mathbb{N}$ such that
\begin{equation}
\label{eq59}0 < \frac{(kp)^{1/p}}{\|u_n^+\|} \leq 1 \quad \mbox{for all } n \geq n_0.
\end{equation}

Let $t_n \in [0,1]$ be such that
\begin{equation}
\label{eq60} \tau^\ast_\lambda(t_n u_n)=\max \{\tau^\ast_\lambda(t u_n): 0 \leq t \leq 1\}.
\end{equation}

It follows from \eqref{eq59} and \eqref{eq60} that
\begin{align*}
\tau^\ast_\lambda(t_n u_n)	& \geq \tau^\ast_\lambda(v_n)\\ & = \frac{c_1 k}{p-1}\|\nabla y_n\|^p_p+ k \left[\int_\Omega \xi(z) |y_n|^p dz + \int_{\partial \Omega} \beta(z) |y_n|^p d \sigma \right]- \int_\Omega J_\lambda(z,v_n)dz\\ & \geq kc_{26} -c_{27} \quad \mbox{for some $c_{26}, c_{27}>0$
 and all $n \geq n_0$ (see Lemmata \ref{Lem5} and \ref{Lem6}).}
\end{align*}	

Since $k>0$ is arbitrary, we infer that
\begin{equation}
\label{eq61}\tau^\ast_\lambda(t_nu_n) \to +\infty \quad \mbox{as } n \to +\infty.
\end{equation}

From the definition of $\tau^\ast_\lambda(\cdot)$ and Corollary \ref{cor3}, we have
$$\tau^\ast_\lambda(u) \leq \tau_\lambda(u) \quad \mbox{for all } u \in W^{1,p}(\Omega).$$

Therefore from \eqref{eq41} we have
\begin{equation}
\label{eq62}\tau^\ast_\lambda(u_n) \leq M_2 \quad \mbox{for all } n \in \mathbb{N}.
\end{equation}

Also, notice that
\begin{equation}
\label{eq63}\tau^\ast_\lambda(0) =0.
\end{equation}

Then \eqref{eq61}, \eqref{eq62}, \eqref{eq63} imply that we can find $n_1 \in \mathbb{N}$ such that 
\begin{equation}
\label{eq64}t_n \in (0,1) \quad \mbox{for all } n \geq n_1.
\end{equation}

It follows from \eqref{eq60} and \eqref{eq64} that 
\begin{align}\nonumber
& \frac{d}{dt} 	\tau^\ast_\lambda (tu_n)\Big|_{t=t_n}=0,\\ \nonumber \Rightarrow \quad &
\langle (\tau^\ast_\lambda)'(t_nu_n), t_nu_n \rangle=0 \quad \mbox{(by the Chain rule)},\\ \nonumber \Rightarrow \quad &
\frac{c_1 }{p-1}\|\nabla (t_n u_n)\|^p_p+ \int_\Omega \xi(z) |t_n u_n|^p dz + \int_{\partial \Omega} \beta(z) |t_n u_n|^p d \sigma = \int_\Omega j_\lambda(z,t_n u_n)(t_n u_n)dz \\ & \leq c_{28} + \int_\Omega [\lambda f(z, t_n u_n^+)+g(z,t_n u_n^+)](t_nu_n^+)dz \label{eq65}\\ \nonumber & \hskip 1cm \mbox{for some $c_{28}>0$
and 
 all $n \geq n_1$ (recall $f,g \Big|_{\Omega \times (-\infty,0]}=0$)}.
	\end{align}

By hypothesis $\widehat{H}_0 \, (i)$ and \eqref{eq64}, we have for $B=\{\lambda\}$
	\begin{align}
		\nonumber & \int_\Omega e_\lambda (z,t_n u_n^+)dz \leq \int_\Omega e_\lambda(z,u_n^+)dz + \|\eta_B\|_1 \leq M_4 + \|\eta_B\|_1  \mbox{ for all $n \geq n_1$ (see \eqref{eq47})},\\ \nonumber \Rightarrow \quad & \int_\Omega [\lambda f(z,t_n u_n^+)+g(z,t_n u_n^+)] (t_n u_n^+)dz\\ & \label{eq66}\leq M_8 + \int_\Omega p[\lambda F(z,t_n u_n^+)+G(z,t_n u_n^+)]dz  \mbox{ for some $M_8>0$
		and
		 all $n \geq n_1$.} 
		\end{align}

Returning to \eqref{eq65} and using \eqref{eq66}, we have 
\begin{align}\nonumber 
& \frac{c_1 }{p-1}\|\nabla (t_n u_n)\|^p_p+ \int_\Omega \xi(z) |t_n u_n|^p dz + \int_{\partial \Omega} \beta(z) |t_n u_n|^p d \sigma \\ \nonumber & \hskip 1cm  - \int_\Omega J_\lambda(z,t_n u_n^+)dz \leq M_9 \quad \mbox{for some $M_9>0$, all $n \geq n_1$ (see \eqref{eq36})},\\ \label{eq67} \Rightarrow \quad & p \tau^\ast_\lambda (t_n u_n) \leq M_9 \quad \mbox{for all } n \geq n_1.
\end{align}	

Comparing \eqref{eq61} and \eqref{eq67} we get a contradiction.

This proves that $\{u_n^+\}_{n \in \mathbb{N}} \subseteq W^{1,p}(\Omega)$ is bounded, therefore 
$$\{u_n\}_{n \in \mathbb{N}} \subseteq W^{1,p}(\Omega) \mbox{ is bounded (see \eqref{eq44})}.$$

We may assume that
		\begin{equation} \label{eq68} u_n \xrightarrow{w} u \mbox{ in } W^{1, p} (\Omega) \mbox{ and } u_n \rightarrow u \mbox{ in } L^p(\Omega) \mbox{ and in } L^p(\partial \Omega). 
	\end{equation}
	
Recall that
\begin{align}\nonumber
& 0 \leq f(z,x) \leq c_{28} [1+|x|^{p-1}] \quad \mbox{for a.a. $z \in \Omega$, all $x \in \mathbb{R}$,
and
 some $c_{28}>0$,}	\\ \label{eq69} \Rightarrow \quad & \int_{\{u_n \geq u_{\lambda_1}\}}f(z,u_n)(u_n-u)dz \to 0 \quad \mbox{(see \eqref{eq68})}.	
\end{align}	

As before, let $c_{29}=\sup_{n \in \mathbb{N}} \|u_n\|_{p^\ast}< +\infty$ (see \eqref{eq68}). Hypotheses $H(g) \, (i),(ii)$ imply that given $\varepsilon >0$, we can find $c_{30}>0$ such that
\begin{equation}
\label{eq70} g(z,x) \leq \frac{\varepsilon}{3 c_{29}^{p^\ast}} x^{p^\ast-1}+c_{30} \quad \mbox{for a.a. $z \in \Omega$
and
 all $x \geq 0$.}
\end{equation}

Suppose that $E \subseteq \Omega$ is measurable. We have 
\begin{align}\nonumber
& \left| \int_E g(z,u_n^+)(u_n-u)dz\right| \\ & \nonumber \leq \int_E |g(z,u_n^+)| |u_n-u|dz \\ & \label{eq71} \leq \frac{\varepsilon}{3 c_{29}^{p^\ast}} \int_\Omega (u_n^+)^{p^\ast -1}|u_n- u| dz + c_{30} \int_E |u_n -u| dz \quad \mbox{(see \eqref{eq70}).} 
\end{align}	

Notice that $(u_n^+)^{p^\ast -1} \in L^{(p^\ast)'}(\Omega)$ (recall $\dfrac{1}{p^\ast}+\dfrac{1}{(p^\ast)'}=1$) and $u_n - u \in L^{p^\ast}(\Omega).$

Using H\"{o}lder's inequality, we have 
\begin{equation}
\label{eq72} \frac{\varepsilon}{3 c_{29}^{p^\ast}} \int_\Omega (u_n^+)^{p^\ast -1}|u_n- u| dz \leq \frac{\varepsilon}{3 c_{29}^{p^\ast}} \|u_n^+\|^{p^\ast -1}_{p^\ast}\|u_n- u\|_{p^\ast} \leq \frac{2 \varepsilon}{3} \quad \mbox{for all } n \in \mathbb{N}.
\end{equation}

Assume that
$$|E|_N \leq \left[\frac{\varepsilon}{6 c_{30}c_{29}}\right]^{(p^\ast)'}
.$$	

Then we have
\begin{equation}
\label{eq73} c_{30}\int_\Omega |u_n - u|dz \leq \frac{\varepsilon}{3} \quad \mbox{for all } n \in \mathbb{N}.
\end{equation}

Returning to \eqref{eq71} and using \eqref{eq72}, \eqref{eq73}, we obtain
\begin{align} \nonumber &
	\left| \int_E g(z,u_n^+)(u_n-u) dz \right| \leq \varepsilon \quad \mbox{for all } n \in \mathbb{N},\\ \nonumber \Rightarrow \quad & \left\{g(\cdot,u_n^+(\cdot))(u_n-u)(\cdot)\right\}_{n \in \mathbb{N}} \subseteq L^1(\Omega) \mbox{ is uniformly integrable,}\\  \Rightarrow \quad & \left\{\chi_{\{u_n \geq u_{\lambda_1}\}}(\cdot)g(\cdot,u_n^+(\cdot))(u_n-u)(\cdot)\right\}_{n \in \mathbb{N}} \subseteq L^1(\Omega) \mbox{ is uniformly integrable.}\label{eq74}
	\end{align}

From \eqref{eq68} and by passing to a subsequence if necessary, we can have
\begin{equation}
\label{eq75} \chi_{\{u_n \geq u_{\lambda_1}\}}(z)g(z,u_n^+(z))(u_n-u)(z)\to 0 \quad \mbox{for a.a. $z \in \Omega$, as $n \to + \infty$}.
\end{equation}

Then \eqref{eq74}, \eqref{eq75} and Vitali's theorem, imply that
\begin{align} \nonumber & \int_{\{u_n \geq u_{\lambda_1}\}}g(z,u_n^+)(u_n-u)dz \to 0,\\ \Rightarrow \quad & \int_{\Omega}j_\lambda(z,u_n^+)(u_n-u)dz \to 0 \mbox{ (see \eqref{eq36}, \eqref{eq44}, \eqref{eq69}).}\label{eq76}
	\end{align}

Therefore if in  \eqref{eq43} we choose $h=u_n -u \in W^{1,p}(\Omega)$, pass to the limit as $n \to +\infty$ and use \eqref{eq44}, \eqref{eq68} and \eqref{eq76}, we obtain 
\begin{align*}&
\lim_{n \to +\infty} \langle A(u_n), u_n-u \rangle =0, \\ \Rightarrow \quad & u_n \to u \mbox{ in $W^{1,p}(\Omega)$ (see Proposition \ref{prop4}),} 	\\ \Rightarrow \quad & \tau_\lambda \mbox{ satisfies the $C$-condition.}
\end{align*}

This proves the claim.

Then \eqref{eq39}, \eqref{eq40} and the claim permit the use of Theorem \ref{T1} (the mountain pass theorem). So, we can find $\widehat{u} \in W^{1,p}(\Omega)$ such that
\begin{equation}
\label{eq77} \widehat{u} \in K_{\tau_\lambda} \subseteq [u_{\lambda_1}) \cap D_+ \quad \mbox{(see \eqref{eq38}),}
\end{equation}
\begin{equation}
\label{eq78} m_\rho^\lambda \leq \tau_\lambda(\widehat{u}) \quad \mbox{(see \eqref{eq39}).}
\end{equation}

From  \eqref{eq36}, \eqref{eq39},  \eqref{eq77}, \eqref{eq78} we conclude that $\widehat{u}$ is the second positive smooth solution of \eqref{PL} ($0 < \lambda < \lambda^\ast$) distinct from $u_0$.
		
\end{proof}

Next, we show that the critical parameter value $\lambda^* >0$ is admissible. In what follows,
 $\varphi_\lambda: W^{1,p}(\Omega) \to \mathbb{R}$ is the energy (Euler) functional for problem \eqref{PL} defined by
$$\varphi_\lambda(u)= \int_\Omega \widehat{G}(\nabla u) dz +\frac{1}{p}\int_\Omega\xi(z) |u|^pdz +$$ $$ \frac{1}{p}\int_{\partial \Omega}\beta(z)|u|^pd\sigma -\int_\Omega[\lambda F(z,u)+G(z,u)]dz,$$
for all $u \in W^{1,p}(\Omega)$.
Evidently, $\varphi_\lambda \in C^1(W^{1,p}(\Omega), \mathbb{R})$ for all $\lambda>0$.
\begin{prop}\label{prop16}
	If hypotheses $\widetilde{H}$ hold, then $\lambda^* \in \mathcal{L}$ and so $\mathcal{L}=(0,\lambda^*]$.
\end{prop}
\begin{proof}
	Let $\{\lambda_n\}_{n \in \mathbb{N}} \subseteq (0,\lambda^*)$ such that $\lambda_n \to (\lambda^*)^-$ as $n \to + \infty$. We can find $u_n \in S(\lambda_n) \subseteq D_+$ for all $n \in \mathbb{N}$. In fact, from Corollary \ref{C12} and the proof of Proposition \ref{P11}, we see that can have $\{u_n\}_{n \in \mathbb{N}}$ increasing and 
	$$\varphi_{\lambda_n} (u_n) <0 \quad \mbox{ for all $n \in \mathbb{N}$}.$$
	Therefore we have 
	\begin{equation}\label{eq79}
		\int_\Omega p\widehat{G}(\nabla u_n) dz+\int_\Omega \xi(z) u_n^p dz +\int_{\partial \Omega}\beta(z) u_n^p d\sigma -p \int_\Omega [\lambda F(z,u_n)+G(z,u_n)]dz <0
	\end{equation}
	for all $n \in \mathbb{N}$. Also we have 
	\begin{align}\label{eq80}
		\nonumber &\langle A(u_n),h \rangle+ \int_\Omega \xi(z) u_n^{p-1}h dz +\int_{\partial \Omega}\beta(z) u_n^{p-1}h d\sigma\\
		= & \int_\Omega [\lambda_n f(z,u_n)+g(z,u_n)]hdz \mbox{ for all $h \in W^{1,p}(\Omega)$, $n \in \mathbb{N}$}.
	\end{align}
	Choosing $h=u_n \in W^{1,p}(\Omega)$ in \eqref{eq80}, we obtain
	\begin{align}\label{eq81}
		\nonumber - \int_\Omega (a(\nabla u_n),\nabla u_n)_{\mathbb{R}^N}dz -& \int_\Omega \xi(z) u_n^{p} dz -\int_{\partial \Omega}\beta(z) u_n^{p} d\sigma\\
		+ & \int_\Omega [\lambda_n f(z,u_n)+g(z,u_n)]u_ndz=0 \mbox{ for  all $n \in \mathbb{N}$}.
	\end{align}
	Adding \eqref{eq79}, \eqref{eq81} and using hypothesis $H(a) \,(iv)$, we have
	\begin{equation}\label{eq82}
		\int_\Omega e_{\lambda_n}(z,u_n) dz <0 \quad \mbox{ for  all $n \in \mathbb{N}$}.
	\end{equation}
	
	From \eqref{eq82} and reasoning as in the proof of Proposition \ref{prop15} (see the part of the proof after \eqref{eq47}), we obtain that 
	$$\{u_n\}_{n \in \mathbb{N}} \subseteq W^{1,p}(\Omega) \quad \mbox{ is bounded}.$$
	So, we may assume that
	\begin{equation}\label{eq83}
		u_n \xrightarrow{w} u^* \mbox{ in $W^{1,p}(\Omega)$ and } u_n \to u^* \mbox{ in $L^p(\Omega)$, and $L^p(\partial \Omega)$}.
	\end{equation}
	In \eqref{eq80} we choose $h=u_n-u^* \in W^{1,p}(\Omega)$ and pass to the limit as $n \to +\infty$. Then 
	\begin{align}\label{eq84}
		\nonumber & \lim_{n \to + \infty}\langle A(u_n),u_n-u^* \rangle =0 \mbox{ (see the part of the proof of Proposition \ref{P11} after \eqref{eq68})},\\
		\Rightarrow \ & u_n \to u^*  \mbox{ in $ W^{1,p}(\Omega)$ (see Proposition \ref{prop4})}.
	\end{align}
	In \eqref{eq80} we pass to the limit as $n \to +\infty$ and use \eqref{eq84}. Then 
	\begin{align*}
		& \langle A(u^*),h \rangle + \int_\Omega \xi(z) (u^*)^{p-1}h dz +\int_{\partial \Omega}\beta(z) (u^*)^{p-1}h d\sigma\\
		= & \int_\Omega [\lambda f(z,u^*)+g(z,u^*)]hdz \mbox{ for all $h \in W^{1,p}(\Omega)$},\\
		\Rightarrow \quad & u^* \in S(\lambda^*) \subseteq D_+.
	\end{align*}
	
	We conclude that $\lambda^* \in \mathcal{L}$ and so $\mathcal{L}=(0,\lambda^*]$.
\end{proof}

\begin{prop}\label{prop17}
	If hypotheses $\widetilde{H}$ hold and $\lambda \to 0^+$, then we can find $u_\lambda \in S(\lambda) \subseteq D_+$ such that $\|u_\lambda\|_{C^1(\overline{\Omega})} \to 0$ as $\lambda \to 0^+$.
\end{prop}
\begin{proof}
	From the proof of Proposition \ref{P10}, we know that for $\lambda >0$ small we can find $u_\lambda \in S(\lambda) \subseteq D_+$ such that
	$$u_\lambda \leq \overline{u}\quad \mbox{ (see \eqref{eq19}, \eqref{eq5})}.$$
	If $\eta \to 0^+$ (see \eqref{eq5}), then $\overline{u}=\overline{u}(\eta) \to 0$ in $C^1(\overline{\Omega})$ and so $u_\lambda \to 0$ in $C^1(\overline{\Omega})$.
\end{proof}

Summarizing, we can state the following theorem describing the dependence of the set of positive solutions on the parameter $\lambda>0$.
\begin{theorem}\label{T18}
	If hypotheses $\widetilde{H}$ hold, then there exists $\lambda^*>0$ such that
	\begin{itemize}
		\item[(a)] if $0< \lambda <\lambda^*$, then problem \eqref{PL} has at least two positive solutions $u_0, \widehat{u} \in D_+$, $u_0 \neq \widehat{u}$;
		\item[(b)] if $\lambda =\lambda^*$, then problem \eqref{PL} has at least one positive solution $u^* \in D_+$;
		\item[(c)] if $\lambda >\lambda^*$, then problem \eqref{PL} has no positive solutions;
		\item[(d)] if $\lambda \to 0^+$, then we can find positive solutions $u_\lambda \in D_+$ such that $\|u_\lambda\|_{C^1(\overline{\Omega})} \to 0$ as $\lambda \to 0^+$.
	\end{itemize}
\end{theorem}

\section*{Acknowledgments}

The first and the second author were supported in part by the Slovenian Research Agency  grants P1-0292,
J1-8131, J1-7025, N1-0064 and N1-0083. We thank the referee for comments and suggestions.

\end{document}